\documentclass[12pt,a4paper]{article}
\usepackage[utf8]{inputenc}
\usepackage{amsfonts,float,subfig}
\usepackage{dsfont}
\usepackage{tikz}
\usepackage{hyperref}
\usepackage{graphicx}

\usepackage{bm}
\usepackage{braket}
\usepackage{amssymb }

\usepackage{amssymb,amsmath,amsthm,epsfig,verbatim,pstricks,ifthen,setspace} 
\usepackage{color}
\newcommand\phaseinth[1]{\int_{\Omega}{#1} \;{\rm{d}} x}

\newcommand{\ppp}{$\mathds{P}$}

\renewcommand{\epsilon}{\varepsilon}

\newcommand{\ul}{u}
\newcommand{\utl}{\partial_t u}

\newcommand\vphihn{\varphi_h^n}
\newcommand\derivinth[1]{\frac{{#1}_h^{n+1} - {#1}_h^{n}}{\Delta t}}

\newcommand\hnorm[1]{\|{#1}\|_{1}}

\newcommand\lnormwH[2]{|\left[{#2}\right]^{\frac{1}{2}}{#1}|_{0}}
\newcommand\lnormH[1]{|{#1}|_{0}}
\newcommand\lnormhH[1]{|{#1}|_{h}}
\newcommand\vphih{\varphi_h}
\newcommand\vphihplus{\varphi^+_h}
\newcommand\vphihminus{\varphi^-_h}
\newcommand\uhl{u_h}
\newcommand\unjzerol{u^n_{j_0}}
\newcommand\unpjzerol{u^{n+1}_{j_0}}
\newcommand\unpil{u^{n+1}_{i}}

\newcommand\unpjl{u^{n+1}_{j}}

\newcommand\unjl{u^{n}_{j}}
\newcommand\timeint[1]{\int_0^T{#1} \;{\rm{d}} t}
\newcommand\ulinehn{u_h^n}
\newcommand\ulinehnp{u_h^{n+1}}
\newcommand\ulinehminus{u_h^-}
\newcommand\ulinehplus{u_h^{+}}
\newcommand\uhtl{\partial_t u_{h}}

\newcommand{\ot}{{\widetilde{\Omega}(t)}}

\newtheorem{Theorem}{Theorem}[section] 

\newtheorem{Lemma}[Theorem]{Lemma}

\newtheorem{Remark}[Theorem]{Remark}

\def\vphi{\varphi}

\newcommand{\bN}{{\mathbb N}}
\newcommand{\bR}{{\mathbb R}}
\newcommand{\bRplus}{{\mathbb R}_{>0}}

\newtheorem{problem}{Problem}[section]
\makeatletter
\newcommand{\setproblemtag}[1]{
	\let\ol\,  {\rm d}theproblem\theproblem
	\renewcommand{\theproblem}{#1}
	\g@ad\,  {\rm d}to@macro\endproblem{
		\ad\,  {\rm d}tocounter{problem}{-1}
		\global\let\theproblem\ol\,  {\rm d}theproblem}
	}
\makeatother

\title{Finite element approximation of a phase field model for tumour growth}

\author{Joe Eyles\footnote{Department of Mathematics,
University of Sussex, Brighton BN1 9HQ, UK},~ 
        Robert N\"{u}rnberg\footnote{Department of Mathematics, University of Trento, Trento, Italy}~~~and~ 
        Vanessa Styles$^*$}

\date{}

\begin{document}
\maketitle

\begin{abstract}
We consider a fully practical finite element approximation of a 
diffuse interface model for tumour growth that takes the form of a 
degenerate parabolic system.
In addition to showing stability bounds for the approximation, 
we prove convergence, and hence existence of a solution to this system in 
two space dimensions.
Several numerical experiments demonstrate the practicality 
and robustness of the proposed method.
\end{abstract}

\section{Introduction}

Mathematical modelling is an important tool in the study and treatment of tumours, and as such is an area of research that continues to attract a great deal of interest, see for example \cite{BellomoLiMaini08,cancer_modelling_overview, cristini_book} and the references therein. 
The mathematical model we consider is one in which the tumour is modelled as a continuum using partial differential equations (PDEs), rather than an individual-based cell model in which a collection of discrete cells undergo stochastic or deterministic behaviour. There is a wealth of literature on continuum tumour growth models, with some of the earliest dating from the 1970s, see \cite{greenspan1,greenspan2} in which the models take the form of free boundary problems, with the free boundary representing the boundary between the interior and exterior of the tumour. Free boundary problems for tumour growth have been successfully approximated using a number of diffuse interface approaches, see for instance \cite{ColliPierluigi;GilardiGianni;Hilhorst, Cristini2003, intro_1,chdtumourModel,GarckeLNS18}. 
In this paper, we analyse the two dimensional version of the diffuse interface model for tumour growth presented in \cite{eks}. 
It approximates the following free
boundary problem, which is one of the simplest mathematical descriptions of the growth and death of a 
tumour:
\begin{subequations} \label{eq:eks}
\begin{align}
\Delta u = 1, & \quad \mbox{ in } \ot, \label{eq_f2_uOnOmega} \\
\nabla u \cdot n + \frac{u}{\alpha} = Q, & \quad \mbox{ on } \Gamma(t), \label{eq_f2_uOnGamma} \\
V = \beta \kappa + \frac{u}{\alpha} , & \quad \mbox{ on } \Gamma(t). \label{eq_f2_vOnGamma} 
\end{align}
\end{subequations}
Here  $\Gamma(t)=\partial\ot$ is an evolving curve in $\bR^2$, with its interior $\ot$ 
representing the tumour cells, and $u$ corresponds to the 
tissue pressure in the tumour. Moreover, $V$ is the velocity of $\Gamma(t)$ 
in the direction of the normal $n$, while $\kappa$ denotes the
curvature of $\Gamma(t)$. The constant $Q \in \bRplus$ represents a
surface source, while $\alpha,\beta \in \bRplus$ are regularisation 
parameters.
The free boundary problem \eqref{eq:eks} was formally derived in \cite{eks} as 
a thin-rim limit of a reaction-diffusion system involving living and dead
tumour cells, a nutrient and a Darcy's law for the velocity and pressure.
The advantage of the limiting problem \eqref{eq:eks} is that the number of
unknowns is reduced to just the free boundary and the pressure.
Mathematically, \eqref{eq:eks} describes the evolution of the closed 
curve $\Gamma(t)$ via forced curvature flow, where the forcing 
depends on the solution of an elliptic PDE in its interior.

A possible diffuse interface approximation of \eqref{eq:eks}
is given as follows. 
Let $\Omega$ be a fixed, bounded domain in $\bR^2$, let $\vartheta\in\{0,1\}$
and let 
$\epsilon\in\bRplus$. Then find a phase field $\varphi_\epsilon:\Omega\to\bR$,
and a pressure $u_\epsilon:\Omega \to \bR$ such that
\begin{subequations} \label{eq:strong}
\begin{alignat}{2}
\epsilon^2\vartheta\partial_t u_\epsilon 
- \nabla \cdot (\zeta(\varphi_\epsilon)\nabla u_\epsilon) 
+\frac1{\epsilon} \delta(\varphi_\epsilon) (\frac{u_\epsilon}\alpha - Q )
+ \zeta(\varphi_\epsilon) & = 0
\qquad && \text{in } \Omega, \label{eq:stronga} \\
\epsilon\partial_t\varphi_\epsilon 
- \epsilon\beta\Delta\varphi_\epsilon 
+ \frac\beta\epsilon\partial W(\varphi_\epsilon) 
- \frac{c_W}{\alpha} u_\epsilon & \ni 0
\qquad && \text{in } \Omega, \label{eq:strongb} \\
\zeta(\varphi_\epsilon)\nabla u_\epsilon \cdot \nu 
= \nabla \varphi_\epsilon \cdot \nu & = 0 \qquad && \text{on }
\partial\Omega, \label{eq:strongc}
\end{alignat}
\end{subequations}
where $\nu$ denotes the outer normal on $\partial\Omega$.
In the above
\begin{equation} \label{eq:zeta}
\zeta(s) = \frac{1+s}{2}~~~\mbox{and}~~~ 
\delta(s) = \frac{2}{\pi} (1 - s^2) \qquad s \in \bR,
\end{equation}
and $W$ denotes the double--obstacle potential 
\begin{equation} \label{eq:obs}
W(s)=\frac12(1-s^2)+I_{[-1,1]}(s) \qquad s \in \bR,
\end{equation}
with $I_{[-1,1]}$ denoting the indicator function 
$$
	I_{[-1, 1]} (s) := 
	\begin{cases}
	+ \infty & \mbox{ for } |s| > 1,\\
	0 & \mbox{ for } |s| \leq 1,
	\end{cases}
$$
see \cite{Blowey_Elliott_double_obs}. Moreover, $c_W = \frac12 \int_{-1}^1 \sqrt{2W(s)}\; {\rm d}s = \frac\pi4$ 
and $\partial W$ denotes the subdifferential of $W$, so that
$\partial W(s) = - s + \partial I_{[-1,1]}(s)$.

The model \eqref{eq:strong}, with $\vartheta=0$, has recently been considered
in \cite{eks}. There it was shown, using formal asymptotic analysis, that
in the sharp interface limit, $\epsilon\to0$, the original problem 
\eqref{eq:eks} is recovered. Moreover, a finite element approximation of
\eqref{eq:strong}, with $\vartheta=0$, was introduced, and several numerical
simulations were presented. 
In this paper, we are interested in the numerical analysis of a suitable
approximation of \eqref{eq:strong}. To this
end, we first of all note that it is a simple matter, on using the results in
\cite{exist_unique_phase}, 
to extend the formal asymptotic  
analysis in \cite{eks} to the case $\vartheta = 1$.
In particular, \eqref{eq:eks} is recovered as the sharp interface limit of
\eqref{eq:strong} with $\vartheta=1$, as $\epsilon\to0$.
Moreover, what makes the analysis of \eqref{eq:strong} challenging is the
presence of the degenerate coefficients $\delta(\varphi_\epsilon)$ and
$\zeta(\varphi_\epsilon)$ in \eqref{eq:stronga}. In the elliptic case,
when $\vartheta=0$, it does not seem to be easily possible to prove convergence
for the finite element approximation introduced in \cite{eks}. Hence, in this
paper, we consider the case $\vartheta=1$, which changes \eqref{eq:stronga} 
to a parabolic problem. 
Although there is no bio-physical basis for the added term
$\epsilon^2\partial_t u_\epsilon$, it yields additional smoothness for
$u_\epsilon$ that enables us to derive a convergence proof. 
However, we note that the model of interest to bio-physicists is
\eqref{eq:eks}, which is not impacted by the additional term in the 
phase field model. 
We remark that
\eqref{eq:strong} with $\vartheta=1$ is closely related to the problem studied
in \cite{digm}. That is why many of the ingredients for our convergence proof
are based on extending the techniques in \cite{digm} to the more complicated
problem \eqref{eq:strong}. 

The remainder of the paper is organised as follows.
In Section~\ref{sec:weak} we present a weak formulation of \eqref{eq:strong} 
and introduce a finite element approximation. We prove the well-posedness and a
maximum principle for the discrete system.
Convergence of the discrete solutions to a weak solution of \eqref{eq:strong} 
is shown in Section~\ref{sec:convergence},
and some numerical simulations are presented in Section~\ref{sec:nr}.
 
We end this section with a few comments about notation.
The $L^2$--inner product on $\Omega$ is denoted by $(\cdot,\cdot)$.
We adopt the standard notation for Sobolev spaces, denoting the norm of
$W^{\ell,p}(\Omega)$ ($\ell \in \bN$, $p \in [1, \infty]$) by
$\|\cdot \|_{\ell,p}$ and the semi-norm by $|\cdot |_{\ell,p}$. For
$p=2$, $W^{\ell,2}(\Omega)$ will be denoted by
$H^{\ell}(\Omega)$ with the associated norm and semi-norm written,
as respectively,  
$\|\cdot\|_{\ell}$ and $|\cdot|_{\ell}$. In addition, we adopt the standard notation $W^{\ell,p}(a,b;X)$
($\ell \in \bN$, $p \in [1, \infty]$, $(a,b)$ an 
interval in $\bR$, $X$ a Banach space) 
for time dependent spaces
with norm $\|\cdot\|_{W^{\ell,p}(a,b;X)}$.
Once again, we write $H^{\ell}(a,b;X)$ if $p=2$. 
Furthermore, $C$ denotes a generic constant independent of 
the mesh parameter $h$ and the time step $\Delta t$, see below.
 
\setcounter{equation}{0}
\section{Weak formulation and finite--element approximation}
\label{sec:weak}

In what follows, we let $\vartheta = 1$, 
fix a value $\epsilon\in\bRplus$ and drop the subscript
$\cdot_\epsilon$ in \eqref{eq:strong} for notational convenience. We also
fix a final positive time $T$.

The obstacle nature of \eqref{eq:obs} means that \eqref{eq:strongb} will lead
to a variational inequality. To this end, we introduce the convex subset
$$K=\{\eta\in H^1(\Omega)| \, |\eta|\leq 1~\mbox{in}~\Omega\}$$ of 
$H^1(\Omega)$. Then a weak formulation of \eqref{eq:strong} is given as
follows.

\ppp\
Let $u(0) = u_0 \in H^1(\Omega)$ and $\varphi(0) = \varphi_0 \in K$.
Then, for $t\in (0,T]$ find $(u(t),\varphi(t)) \in H^1(\Omega)\times K$ 
such that 
\begin{subequations} \label{eq:weak}
\begin{align} \label{equation_phase_weak_u}
&\varepsilon^2 \left( \partial_t u , \eta \right) +\left( \zeta(\varphi) \nabla {u}, \nabla \eta \right)
+ \frac{1}{\epsilon\alpha} \left(\delta(\varphi) {u}, \eta\right)  
	= \frac{Q}{\varepsilon}\left(\delta(\varphi), \eta \right) -\left(\zeta(\varphi), \eta \right)
	\quad \forall \eta \in H^1(\Omega),\\
& \varepsilon \left(\partial_t \varphi ,\rho - \varphi\right) 
+ \varepsilon \beta    \left( \nabla \varphi ,\nabla (\rho - \varphi)\right) 
 -  \frac{\beta}{\varepsilon}      \left(\varphi, \rho - \varphi \right) 
 \geq\frac{c_W}{\alpha}   \left(u, \rho - \varphi\right)
~~\forall \rho \in K. \label{equation_phase_weak_v}
\end{align} 
\end{subequations}

We consider the finite element approximation of \ppp\ under the following
assumptions on the mesh:
\begin{itemize} 
\item[(A)] Let $\Omega$ be a polygonal domain. Let 
 $\{{\cal T}_h\}_{h>0}$ be a quasi-uniform family of partitionings of $\Omega$ 
 into disjoint open simplices $\sigma$ with $h_{\sigma}:={\rm diam}(\sigma)$ 
 and $h:=\max_{\sigma \in {\cal T}_h}h_{\sigma}$, so that 
 $\overline{\Omega}=\cup_{\sigma\in{\cal T}_h}\overline{\sigma}$.
In addition, it is assumed that all simplices $\sigma \in {\cal T}_h$ are 
acute. 
\end{itemize}
Associated with $\mathcal{T}_h$ is the finite element space
\[
S_h := \Set{ \eta_h \in C^0(\overline\Omega) |  \eta_h |_{\sigma} 
\mbox{ is linear } \forall \sigma \in \mathcal{T}_h} .
\]
We also introduce 
\[
	K_h := S_h \cap K.
\]
Let $J$ be the set of nodes of ${\cal T}_h$ and $\{p_{j}\}_{j \in J}$ the 
coordinates of these nodes. 
Let $\{\chi_{j}\}_{j\in J}$, 
be the standard basis functions for $S_h$; that is $\chi_{j} \in S_h$ 
and $\chi_j(p_{i})=\delta_{ij}$ for all $i,j \in J$.
We denote by $I_h : C^0(\overline\Omega) \to S_h$ the Lagrange 
interpolation operator onto $S_h$. 
We denote the discrete $L^2$--inner product on $C^0(\overline\Omega)$ by
$$
(u,v)_h:= (I_h(u v),1) \qquad u,v \in C^0(\overline\Omega),
$$ 
and we set $\lnormhH{v}^2 := (v,v)_h$. 

In addition to $\mathcal{T}_h$, let 
$0= t_0 < t_1 < \ldots < t_{N-1} < t_N = T$ be a
partitioning of $[0,T]$ into uniform time steps 
$t_n - t_{n-1} = \Delta t = T/N$, $n=1,\ldots, N$.
Then we consider the following finite element approximation of \ppp.

\ppp$_h$\ 
Let $u_h^0 := I_h u_0$ and $\vphi_h^0 := I_h \varphi_0$. Then, for
$n = 0 , \ldots , N-1$, 
find $u_h^{n+1}\in S_h$ and $\vphi_h^{n+1} \in K_h$ such that 
\begin{subequations} \label{eq:fea}
\begin{align}
\label{equation_phase_FEM_u}
&  \frac{\varepsilon^2}{\Delta t}\left( u_{h}^{n+1}-{u}_h^n, \eta_h \right)_h +\big( \zeta(\varphi^{n}_h) \nabla {u}_h^{n+1}, \nabla \eta_h \big)
+\frac{1}{\epsilon\alpha} \big(\delta(\varphi_h^n){u}_h^{n+1},\eta_h\big)_h  
\nonumber\\ & \qquad \qquad 
=\frac{Q}{\varepsilon} \big(\delta(\varphi_h^n),\eta_h\big)_h-\big(\zeta(\varphi_h^n),\eta_h\big)_h~~\forall \eta_h \in S_h \\
& \frac\epsilon{\Delta t}\left( \vphi_{h}^{n+1}-\vphi_h^n, 
\rho_h - \varphi_h^{n+1} \right)_h 
	+ \epsilon\beta \left(\nabla \varphi_h^{n+1} , \nabla \rho_h - \nabla \varphi_h^{n+1}\right)  \nonumber \\
	&\hspace{20mm} - \frac{\beta}{\epsilon} \left(  \varphi_h^{n+1} , \rho_h  - \varphi_h^{n+1}\right)_h
	\geq 
	\frac{c_W}{\alpha} \left( u_h^{n+1}, \rho_h - \varphi_h^{n+1}\right)_h
	 ~~\forall \rho_h \in K_h.\label{equation_phase_FEM_v}
\end{align}
\end{subequations}
We note that the discretisation in \eqref{eq:fea} is chosen such that the
linear system \eqref{equation_phase_FEM_u} decouples from the variational
inequality \eqref{equation_phase_FEM_v}. Hence in practice we can first solve
\eqref{equation_phase_FEM_u} to obtain $u_{h}^{n+1}$, and then find
$\vphi_h^{n+1} \in K_h$ solving \eqref{equation_phase_FEM_v}. 

In what follows, we will make use of the standard inequality
\begin{equation}
\label{eq_h_norm_bounded}
	\lnormH{v_h}^2 \leq \lnormhH{v_h}^2 \leq 4 \lnormH{v_h}^2
\qquad v_h \in S_h.
\end{equation}

\begin{Lemma}\label{lem_exist}
Let $u_h^n \in S_h$ and $\varphi_h^n \in K_h$. Then for $\Delta t<\frac{\varepsilon^2}\beta$ there exists
a unique solution $(u_h^{n+1}, \varphi_h^{n+1}) \in S_h \times K_h$ to 
\eqref{eq:fea}.
\end{Lemma}
\begin{proof}
The existence of a unique solution $u_{h}^{n+1}$ to 
\eqref{equation_phase_FEM_u} follows immediately from the fact that the 
system is linear, symmetric and positive definite.
In addition, there exists a $\varphi_h^{n+1} \in K_h $ solving 
(\ref{equation_phase_FEM_v}) since this is the Euler-Lagrange variational 
inequality of the minimisation problem:
\[
\min_{z_h \in K_h} \left\{ 
\frac{\epsilon \beta }{2} |\nabla z_h|_{0}^2 + \frac{\epsilon}{2 \Delta t} | I_h(z_h - \vphihn) | _{0}^2  
+ \frac{\beta}{2\epsilon} (1-z_h^2,1)_h -  \frac{c_W}{\alpha} \left(u_{h}^{n+1},  z_h \right)_h \right\},
\]
for which the existence of a minimiser can be shown by a standard minimisation argument. For uniqueness 
we assume there exist two solutions $\varphi_{h,1}^{n+1}$ and 
$\varphi_{h,2}^{n+1}$ and set $\theta=\varphi_{h,1}^{n+1}-\varphi_{h,2}^{n+1}$.
Taking 
$\eta_h=\varphi_{h,2}^{n+1}$ in the variational inequality 
(\ref{equation_phase_FEM_v}) for $\varphi_{h,1}^{n+1}$ and vice versa, 
then subtracting one of the resulting inequalities from the other  we obtain 
$$
 \frac{\epsilon^2}{\Delta t}|\theta|^2_h+ \epsilon^2\beta |\nabla\theta|_0^2  
- \beta |\theta|^2_h\leq 0~~\Rightarrow~~ \left(\frac{\epsilon^2}{\Delta t}-\beta\right)|\theta|^2_h+ \epsilon^2\beta |\nabla\theta|_0^2  \leq 0,
$$	
which yields uniqueness for $\Delta t<\frac{\varepsilon^2}\beta$.
\end{proof}

\setcounter{equation}{0}
\section{Convergence of the finite element scheme}
\label{sec:convergence}

In this section, which makes use of many of the techniques in \cite{digm}, 
we prove that as $h \rightarrow 0$ the solution of the finite element scheme 
$\mathds{P}_h$ converges to the solution of \ppp.
Here we fix $\epsilon$, and assume that $\Delta t \leq Ch^2$. 
All limits are taken as $h \rightarrow 0$ (and thus $\Delta t \rightarrow 0$). 
We introduce the notation
\begin{align*}
\vphih(t) & := 
\frac{t-t_n}{\Delta t} \varphi_h^{n+1} + \frac{t_{n+1}-t}{\Delta t} \vphihn, 	
\quad t \in (t_n,t_{n+1}], \\
\vphihplus(t) & := \varphi_h^{n+1} , \quad \vphihminus(t) := \vphihn, 
\quad t \in (t_n,t_{n+1}], 
\end{align*}
and similarly for $u_h(t)$, $u_h^+(t)$ and $u_h^-(t)$. 
Furthermore we use $\cdot^{(\star)}$ to denote an expression with or without
the superscript $\star$.

Rewriting \eqref{eq:fea} using the above notation gives, for $n=0,\ldots,N-1$
and for $t\in (t_n,t_{n+1}]$, 
\begin{subequations} \label{eq:convfea}
\begin{align} \label{equation_phase_FEM_ua}
&\varepsilon^2\left(\partial_t  u_{h}, \eta_h \right)_h 
+\big( \zeta(\varphi^{-}_h) \nabla {u}_h^+, \nabla \eta_h \big)
+\frac{1}{\epsilon\alpha} \big(\delta(\varphi_h^-){u}_h^+,\eta_h\big)_h 
=\frac{Q}{\epsilon} \big(\delta(\varphi_h^-),\eta_h\big)_h
-\big(\zeta(\varphi_h^-),\eta_h\big)_h, \\ 
& 
\epsilon \left( \partial_t \vphi_h, \rho_h - \varphi_h^{+} \right)_h 
+ \epsilon \beta \left(\nabla \varphi_h^{+} , \nabla \rho_h - \nabla \varphi_h^{+}\right) 
- \frac{\beta}{\epsilon} \left( \varphi_h^{+}, \rho_h - \varphi_h^{+}\right)_h 
\geq \frac{c_W}{\alpha} 
\left( u_h^+, \rho_h - \varphi_h^{+}\right)_h, 
\label{equation_phase_FEM_va}
\end{align}
\end{subequations}
for all $\eta_h \in S_h$ and $\rho_h \in K_h$.

\begin{Lemma}
\label{lemma_phase_digm_4_1_u}
We have
\begin{equation}
\label{equation_digm_4_1_sup_t_u}
\sup_{t \in [0, T]} \lnormH{ \uhl(t)}^2 \leq C , 
\end{equation}
\begin{equation}
\label{equation_digm_4_1_L2_bounds_u}
\timeint{ \left( \lnormwH{\nabla \ulinehplus}{\zeta(\vphi_h^{-})}^2 + \lnormwH { \ulinehplus }{\delta(\vphi_h^-)}^2\right)} \leq C ,
\end{equation}
and
\begin{equation}
\label{equation_digm_4_1_L2_bound_u_t}
	\timeint{ \lnormH{ \ulinehplus - \ulinehminus }^2 } \leq C \Delta t.
\end{equation}
\end{Lemma}
\begin{proof}
Setting $\eta_h = \ulinehnp$ in 
(\ref{equation_phase_FEM_u}) 
we obtain from \eqref{eq:zeta} and 
$\vphi_h^n \in K$ that
\begin{align*}
 \frac{\epsilon^2}{2} \lnormhH{\ulinehnp - \ulinehn}^2 
 &+ \frac{\epsilon^2}{2} \left( \lnormhH{\ulinehnp  }^2 - \lnormhH{\ulinehn }^2 \right)   + \Delta t \lnormwH{ \nabla \ulinehnp }{\zeta(\vphi_h^{n})}^2 
+ \frac{\Delta t}{\epsilon\alpha} 
\lnormhH{[\delta(\vphi_h^{n})]^\frac12 \ulinehnp }^2
\nonumber \\&\qquad 
= \frac{Q \Delta t}{\epsilon} \left( \delta(\vphi_h^{n}), \ulinehnp \right)_h 
- \Delta t \left( \zeta(\vphi_h^{n}), \ulinehnp  \right)_h \leq C\Delta t \lnormhH{\ulinehnp  }^2
\end{align*}
Summing over $n = 0 , \ldots, N-1$, using a discrete Gronwall inequality and recalling \eqref{eq_h_norm_bounded} gives the required results.
\end{proof}

\begin{Lemma}
\label{lemma_phase_digm_4_1_varphi}
We have that
\begin{equation}
\label{equation_digm_4_1_varphi_leq_c}
	\sup_{t \in [0, T]} \lnormH{ \nabla \varphi_h(t) }^2 + \timeint{\lnormH{\partial_t \varphi_{h}}^2} \leq C ,
\end{equation}
and
\begin{equation}
\label{equation_digm_4_1_varphi_leq_c_dt}
	\timeint{\lnormH{ \nabla (\vphihplus - \vphihminus)}^2} \leq C \Delta t .
\end{equation} 

\end{Lemma}
\begin{proof}
Choosing $\rho_h = \vphihn \in K_h$ in (\ref{equation_phase_FEM_v}), 
re-arranging and using the elementary identity 
\begin{equation*}
 2r(r-s) = (r^2-s^2)+(r-s)^2 \quad \forall\ r,s \in \bR,
\end{equation*}
gives
\begin{align}
\Delta t \lnormhH{\derivinth{\varphi}}^2 &+ \frac{\beta}{2} \left( \lnormH{\nabla \varphi_h^{n+1} }^2 - \lnormH{\nabla \vphihn }^2 \right) 
+ \frac{\beta}{2} \lnormH{\nabla (\varphi_h^{n+1} - \vphihn)}^2 \nonumber \\
& + \frac{\beta}{2 \epsilon^2} \left(  \lnormhH{\varphi_h^{n+1}}^2 - \lnormhH{\varphi_h^{n}}^2 \right) 
+ \frac{\beta}{2 \epsilon^2} \lnormhH{\varphi_h^{n+1} - \vphihn}^2  
\leq  \frac{c_W}{ \epsilon \alpha} \left( u_{h}^{n}, \varphi_h^{n+1}-\varphi_h^n\right)_h . \nonumber
\end{align}
Applying Young's inequality to the right hand side yields
\begin{align*}
\frac{\Delta t}2 \lnormhH{\derivinth{\varphi}}^2 &+ \frac{\beta}{2} \left( \lnormH{\nabla \varphi_h^{n+1} }^2 - \lnormH{\nabla \vphihn }^2 \right)  
+ \frac{\beta}{2} \lnormH{\nabla (\varphi_h^{n+1} - \vphihn)}^2 \\
&\leq \frac{\Delta t}2 (\frac{c_W}{\epsilon \alpha})^2 \lnormhH{u_{h}^{n}}^2 
\,+\,\frac{\beta}{2 \epsilon^2} \left(  \lnormhH{\vphi_h^{n+1}}^2 - \lnormhH{\varphi_h^{n}}^2 \right). 
\end{align*}
Summing over $n = 0, \ldots, N-1$ and noting (\ref{eq_h_norm_bounded}), 
$\vphihn \in K_h$ and (\ref{equation_digm_4_1_sup_t_u}) 
yields (\ref{equation_digm_4_1_varphi_leq_c}) and 
(\ref{equation_digm_4_1_varphi_leq_c_dt}).
\end{proof}

\begin{Lemma}
\label{lemma_phase_digm_4_2_u}
We have
\[
	\timeint{\|\uhtl \|^2_{(H^1(\Omega))'}} \leq C .
\]
\end{Lemma}
\begin{proof}
Let $\psi \in H^1({\Omega})$ be arbitrary, and let $J_h \psi \in S_h$ 
be its $L^2$-projection, defined via
\[
(\psi,v_h) = (J_h \psi, v_h)_h \quad \forall v_h \in S_h.
\]
We have $(v_h, \psi)_{((H^1)', H^1)} = (v_h, \psi) = (v_h, J_h\psi)_h$
for $v_h \in S_h$, and and so setting $\eta_h = J_h \psi$ in 
(\ref{equation_phase_FEM_u}) yields
\begin{align*}
\frac{\epsilon^2}{\Delta t} ( u_h^{n+1} - u_h^n, \psi)_{((H^1)', H^1 )} =
&- (\zeta(\vphi_h^{n}) \nabla \ulinehnp, \nabla (J_h \psi))
- \frac{1}{\epsilon\alpha} \left( \delta(\vphi_h^{n}) \ulinehnp, J_h \psi 
\right)_h \\ &
+ \frac{Q}{\epsilon} \left( \delta(\vphi_h^{n}), J_h \psi \right)_h
- \left( \zeta(\vphi_h^{n}), J_h\psi \right)_h . 
\end{align*}

It can be shown that $\hnorm{J_h \psi} \leq C \hnorm{\psi} $, for all $\psi \in H^1(\Omega)$ (see, for example, \cite{digm_ref_1_bar_blo_gar}). Using this fact, together with (\ref{equation_digm_4_1_sup_t_u}) and the bounds on the $L^\infty$ norms of $\delta(\vphi_h^{n})$ and $\zeta(\vphi_h^{n})$, we have the following for all $\psi \in H^1({\Omega})$ 
\begin{align*}
\frac{\epsilon^2}{\Delta t} ( u_h^{n+1} - u_h^n, \psi)_{((H^1)', H^1 )} 
& \leq C 
\lnormwH{\nabla u_h^{n+1} }{\zeta(\vphi_h^{n})} \lnormH{\nabla (J_h \psi)}
+ C \lnormH{J_h \psi} \nonumber \\
& \leq C( 1 +  
\lnormwH{\nabla u_h^{n+1} }{\zeta(\vphi_h^{n})} ) \hnorm{ \psi}.
\end{align*}
We conclude that
\[
	\| \frac{u_h^{n+1} - u_h^n}{\Delta t} \|_{(H^1(\Omega))'} \leq 
C( 1 +  \lnormwH{\nabla u_h^{n+1} }{\zeta(\vphi_h^{n})} ) .
\]
Squaring, multiplying by $\Delta t$, and summing from $n = 0, \ldots , N-1$ 
yields the required result, on noting \eqref{equation_digm_4_1_L2_bounds_u}.
\end{proof}

{From} Lemmas~\ref{lemma_phase_digm_4_1_varphi} and 
\ref{lemma_phase_digm_4_2_u} and the Aubin--Lions lemma,  
we have that there exist 
functions $\varphi \in L^\infty(0,T;H^1(\Omega)) \cap
H^1(0,T; (H^1(\Omega))')$, $u \in L^\infty(\Omega \times(0,T)) \cap
H^1(0,T; (H^1(\Omega))')$
and $F \in L^2(0,T; [L^2(\Omega)]^2)$
such that, after possibly re-indexing from 
subsequences, it holds as $h \rightarrow 0$ that
\begin{subequations} \label{eq:conv}
\begin{align}
	\varphi_{h}^{(\pm)} \overset{*}{\rightharpoonup} \varphi \quad & \mbox{ in }  L^\infty ( 0, T;  H^1(\Omega) )  , 
\label{eq_digm_all_convergences_1} \\
	 \partial_t \varphi_{h} \rightharpoonup \partial_t\varphi \quad & \mbox{ in }  L^2 ( 0, T; L^2(\Omega) ) ,
\label{eq_digm_all_convergences_2} \\
	\varphi_{h}^{(\pm)} \rightarrow \varphi \quad & \mbox{ in }  L^2 ( 0, T; L^2(\Omega) ) , 
\label{eq_digm_all_convergences_3} \\
\uhl^{(\pm)} \rightharpoonup \ul \quad & \mbox{ in }  L^\infty (0,T;L^2( \Omega) ) ,
\label{eq_digm_all_convergences_4} \\
	\partial_t u_{h} \rightharpoonup \partial_t u \quad & \mbox{ in }  L^2 ( 0, T; (H^1(\Omega))' ) ,
\label{eq:uht} \\
	\zeta(\vphi_h^{-}) \nabla u_h^+ \rightharpoonup F \quad & \mbox{ in } L^2(0,T; [L^2(\Omega)]^2) .
\label{digm_4_3_conv_2}
\end{align}
\end{subequations}
Here (\ref{digm_4_3_conv_2}) follows directly from 
(\ref{equation_digm_4_1_L2_bounds_u}) and the bound on the $L^\infty$ norm of 
$\zeta(\vphi_h^{-})$. The function $F$ will be identified later, 
see Lemma~\ref{lemma_phase_digm_4_4_zeta}.

\begin{Theorem}
\label{thm_digm_theorem1_varphi}
The functions $\varphi$ and $u$ in \eqref{eq:conv} satisfy 
\eqref{equation_phase_weak_v}.
\end{Theorem}
\begin{proof}
Using (\ref{eq_digm_all_convergences_1})--(\ref{eq_digm_all_convergences_4}) we show that (\ref{equation_phase_FEM_v}) converges to (\ref{equation_phase_weak_v}). 
Starting with (\ref{equation_phase_FEM_v}), we multiply by an arbitrary $\psi \in C^\infty_0 (0, T), \psi \geq 0$, 
and integrate over $t \in (0, T)$, to obtain
\begin{align*}
& \underbrace{ \frac{\epsilon}{\Delta t} \timeint{\psi \left( \varphi_h^{+} - \varphi_h^{-} , \rho_h - \varphi_h^{+} \right)_h }}_{(1)}
	+ \underbrace{ \epsilon\beta\timeint{\psi  \left(\nabla \varphi_h^{+} , \nabla (\rho_h - \varphi_h^{+}) \right) }}_{(2)} \nonumber \\
	&- \underbrace{ \frac{\beta}{\epsilon} \timeint{\psi  \left( \varphi_h^{+} ,\rho_h  - \varphi_h^{+} \right)_h }}_{(3)}
	-  
	\underbrace{ \frac{c_W }{ \alpha} \timeint{\psi  \left( u_h^+ ,\rho_h - \varphi_h^{+} \right)_h }}_{(4)}
	\geq  0. 
\end{align*}
Since $\rho \in K$, there exists a sequence $\rho_h \in K_h$ such that $\rho_h \rightarrow \rho$ in $H^1(\Omega)$ as $h \rightarrow 0$.

For all but the second integral we use the well known inequality,   
\begin{equation}\label{numerical_int}
      |(\eta_h,\chi)-(\eta_h,\chi)_h|  \leq Ch |\eta_h|_{1} |\chi |_{0}~~~\forall~\eta_h,\chi\in S_h.
\end{equation}
For  $(1)$ we note that $\partial_t \vphi_{h} = \frac{\vphi_h^{+} - \vphi_h^-}{\Delta t}$ on $(t_{n+1}, t_{n})$, and we can thus apply (\ref{eq_digm_all_convergences_2}) and (\ref{eq_digm_all_convergences_3}). For  $(2)$ we use (\ref{eq_digm_all_convergences_1}), (\ref{eq_digm_all_convergences_3}), and the weak lower semi-continuity of the $L^2$ norm (which gives that $- | \nabla \vphi |^2_2 \geq - \liminf_{h \rightarrow 0} | \nabla \vphi_h^+ |^2_2$). 
For  $(3)$ we use (\ref{eq_digm_all_convergences_3}). Finally, for  $(4)$, we use (\ref{eq_digm_all_convergences_3}) and (\ref{eq_digm_all_convergences_4}). This yields
\begin{align*}
\epsilon  \timeint{\psi \left( \partial_t \varphi, \rho - \varphi \right) }
	&+  \epsilon\beta \timeint{ \psi \left( \nabla \varphi , \nabla (\rho - \varphi) \right) } 
	- \frac{\beta}{\epsilon} \timeint{\psi \left( \varphi ,\rho - \varphi \right) }
	  \nonumber\\ 
	  &\qquad\qquad-\frac{c_W}{\alpha} \timeint{\psi \left( \ul ,\rho - \varphi\right) }
	  \geq  0  . \nonumber
\end{align*}
As $\psi \geq 0$ is arbitrary, this gives us the result in the limit as $h \rightarrow 0$.
\end{proof}

\begin{Lemma}
\label{lemma_phase_digm_strong_convergence_weights}
We have
\begin{subequations}
\begin{align}
\label{equation_phase_digm_convergence_of_weights_zeta}
\zeta(\vphi_h^{(\pm)}) \rightarrow \zeta(\vphi) \quad &\mbox{ in } L^2(0,T; L^2(\Omega)),\\
\label{equation_phase_digm_convergence_of_weights_delta}
\delta(\vphi_h^{(\pm)}) \rightarrow \delta(\vphi) \quad &\mbox{ in } L^2(0,T; L^2(\Omega)),\\
\label{equation}
\delta(\vphi_h^{(-)})u_h^{(+)}\rightharpoonup\delta(\vphi)u \quad &
\mbox{ in } L^2(0,T; L^2(\Omega)).
\end{align}
\end{subequations}
\end{Lemma}
\begin{proof}
The statements (\ref{equation_phase_digm_convergence_of_weights_zeta}) 
and \eqref{equation_phase_digm_convergence_of_weights_delta} 
follow trivially from $\varphi_h \in K$ and (\ref{eq_digm_all_convergences_3}).  
Then (\ref{equation}) follows by combining (\ref{eq_digm_all_convergences_4}) and (\ref{equation_phase_digm_convergence_of_weights_delta}).
\end{proof}
The following lemma provides us the necessary $L^\infty(0,T;L^\infty(\Omega))$ bound on $u_h^{n+1}$ that we require to prove the convergence results  
in the subsequent lemma, Lemma \ref{lemma_phase_digm_4_3_deltazeta}. 

\begin{Lemma}
\label{lemma_phase_max_on_u_h}
Let $(u_h^n,\varphi_h^n) \in S_h\times K_h$ and let $u_h^{n+1}\in S_h$ be 
the unique solution to \eqref{equation_phase_FEM_u}. Then it holds that
\[
- | u_h^0 |_{0,\infty} - \epsilon^{-2}T \leq u_h^{n+1}
\leq \max \left( \alpha Q, | u_h^0 |_{0,\infty} \right) 
\quad \text{in } \overline\Omega.
\]
\end{Lemma}
\begin{proof}
Throughout this proof, 
we use the shorthand notations $u^{n+1}_j = u_h^{n+1}(p_j)$, and 
$\varphi_j^n = \varphi_h^n(p_j)$ for $j\in J$.
We first use an inductive argument to prove the maximum bound.  
We assume that 
$\max_{j\in J} \unjl \leq \max \left( \alpha Q, | u_h^0 |_{0,\infty} \right)$,
which clearly holds for $n=0$, and set 
$\unpjzerol := \max_{j \in J} \unpjl = \max_{\overline\Omega} u_h^{n+1}$.
Since $\mathcal{T}_{h}$ is acute, we have 
$\nabla \chi_i |_\sigma \cdot \nabla \chi_{j} |_\sigma \leq 0$ for $i \neq j$, 
and hence 
\begin{align} 
\big( \zeta(\varphi^{n}_h) \nabla {u}_h^{n+1}, \nabla \chi_{j_0} \big)
& = \sum_{\sigma \in \mathcal{T}_h} \sum_{i\in J} 
\unpil \nabla \chi_i |_\sigma \cdot \nabla \chi_{j_0} |_\sigma 
\int_\sigma \zeta(\vphi_h^{n})\;{\rm{d}}x\nonumber \\ & 
\geq \sum_{\sigma \in \mathcal{T}_h} \unpjzerol \sum_{i \in J}  \nabla \chi_i |_\sigma \cdot \nabla \chi_{j_0} |_\sigma 
\int_\sigma \zeta(\vphi_h^{n})\;{\rm{d}}x= 0. \label{digm_acute_mesh_max}
\end{align}
Hence choosing $\eta_h = \chi_{j_0}$ in \eqref{equation_phase_FEM_u} implies
\begin{align}
\frac{\epsilon^2}{\Delta t} (\unpjzerol - \unjzerol) (1,\chi_{j_0})
& \leq \left( \frac{1}{\epsilon} \delta(\vphi_{j_0}^{n}) (Q - \frac{\unpjzerol}{\alpha} ) - \zeta(\vphi_{j_0}^{n}) \right) (1,\chi_{j_0}) \nonumber \\
& \leq  \frac{1}{\epsilon} \delta(\vphi_{j_0}^{n})(Q - \frac{\unpjzerol}{\alpha } ) (1,\chi_{j_0}).
\nonumber 	
\end{align}
If $\unpjzerol > \max \left( \alpha Q, | u_h^0 |_{0,\infty} \right)$, 
then we have $\unpjzerol < \unjzerol \leq \max 
\left( \alpha Q, | u_h^0 |_{0,\infty} \right)$, 
which is a contradiction,
and thus $\unpjzerol \leq \max \left( \alpha Q, | u_h^0 |_{0,\infty}\right)$.\\
For the minimum bound, we again use an inductive argument. 
We assume that 
$\min_{j\in J}\unjl \geq - \frac{n \Delta t}{\epsilon^2} - |u_h^0|_{0,\infty}$ 
and set $\unpjzerol := \min_{j \in J } \unpjl$. 
Choosing $\eta_h = \chi_{j_0}$ in (\ref{equation_phase_FEM_u}) yields
\begin{equation} \label{eq:min1}
\frac{\epsilon^2}{\Delta t} (\unpjzerol - \unjzerol) (1,\chi_{j_0})
\geq
\left( \frac{1}{\epsilon} \delta(\vphi_{j_0}^{n}) 
(Q - \frac{\unpjzerol}{\alpha} ) - \zeta(\vphi_{j_0}^{n}) \right) 
(1,\chi_{j_0}),
\end{equation}
where we used a similar argument to (\ref{digm_acute_mesh_max}).
If $\unpjzerol \geq 0$ we have nothing to prove, and so it is sufficient
to consider the case $\unpjzerol < 0$. Then 
$Q - \frac{\unpjzerol}{\alpha} > 0$, which together with \eqref{eq:min1}
and \eqref{eq:zeta} implies that 
\[
\frac{\epsilon^2}{\Delta t} (\unpjzerol - \unjzerol) 
\geq
- \zeta(\vphi_{j_0}^{n}) \geq - 1.
\]
Hence we have that
\[
\unpjzerol \geq -\frac{\Delta t}{\epsilon^2} + \unjzerol
\geq - \frac{(n+1)\Delta t}{\epsilon^2} - |u_h^0|_{0,\infty} ,
\]
which proves the inductive assumption and hence yields the desired lower bound.
\end{proof}

\begin{Lemma}
\label{lemma_phase_digm_4_3_deltazeta}
We have $\delta(\vphi) \ul, \zeta(\vphi) \ul \in L^2(0,T; H^1(\Omega))$ and 
\begin{subequations} \label{eq:lem36}
\begin{align}
&\delta(\vphi_h) \uhl \rightarrow \delta(\vphi) \ul \quad \mbox{ in } L^2(0,T; L^2(\Omega)), \label{eq:lem36a} \\ 
&\zeta(\vphi_h) \uhl \rightarrow \zeta(\vphi) u \quad \mbox{ in } L^2(0,T; L^2(\Omega)) \label{eq:lem36b}.
\end{align}
\end{subequations}
\end{Lemma}
\begin{proof}
We first obtain a bound on $\delta(\vphi_h) \uhl $ in $L^2(0,T; H^1(\Omega))$ and then we obtain a bound on $\partial_t(\delta(\vphi_h) \uhl )$ in $L^2(0,T; (W^{1,p}(\Omega))')$, for $p \in(2,\infty)$. 
To bound $\lnormH{\nabla (\delta(\vphi_h) \uhl)}^2$ we first note that from Lemma~\ref{lemma_phase_digm_4_1_varphi} we have $\lnormH{\nabla \delta(\vphi_h)}^2 \leq C$, and thus noting Lemma \ref{lemma_phase_max_on_u_h}, we have
\begin{align*}
	\lnormH{\nabla (\delta(\vphi_h) \uhl)}^2 
	& \leq \lnormH{\nabla \delta(\vphi_h)}^2 |\uhl|_{0,\infty}^2 + |\delta(\vphi_h)\nabla \uhl |^2_0  
\leq C + |\delta(\vphi_h)\nabla \uhl |^2_0 .
\end{align*}
Noting that, for $t \in (t_n, t_{n+1})$,
\[
\delta(\vphi_h) 
\leq \delta(\vphi_h^{n}) + C|\vphi_h^{n+1} - \vphi_h^{n} |,
\]
and using the bound on the $L^\infty$ norm of $\delta(\varphi_h)$, we have, 
for $t \in (t_n, t_{n+1})$, 
\begin{align*}
| \delta(\vphi_h) \nabla \uhl |_0^2 & \leq 
C |\delta(\vphi_h^{n}) \nabla \uhl |^2_0
+ C |(\varphi_h^{n+1} - \vphihn) \nabla \uhl|^2_0 
\nonumber \\ & 
\leq C |\delta(\vphi_h^{n})  \nabla \ulinehnp |^2_0
+ C |\delta(\vphi_h^{n}) \nabla (\ulinehnp - \ulinehn) |^2_0 
+ C |(\varphi_h^{n+1} - \vphihn) \nabla \uhl|^2_0 \nonumber \\ & 
\leq C \lnormwH{\nabla \ulinehnp}{\delta(\vphi_h^{n})}^2 + 
C |\nabla(\ulinehnp - \ulinehn)|_{0}^2+ C |\nabla \uhl|^2_{0,\infty} 
\lnormH{\varphi_h^{n+1} - \vphihn}^2 , \nonumber \\ & 
\leq 
C \lnormwH{\nabla \ulinehnp}{\delta(\vphi_h^{n})}^2 + Ch^{-2}( \lnormH{\ulinehnp - \ulinehn}^2+ |\uhl|^2_{0,\infty} \lnormH{\varphi_h^{n+1} - \vphihn}^2) 
\nonumber \\ & 
\leq C \lnormwH{\nabla \ulinehnp}{\zeta(\vphi_h^{n})}^2 + Ch^{-2}( \lnormH{\ulinehnp - \ulinehn}^2+ \lnormH{\varphi_h^{n+1} - \vphihn}^2) ,
\end{align*}
where we have  used the fact that 
$\delta(s) = \frac2\pi (1 + s) (1 - s) = \frac4\pi (1 - s) \zeta(s)$,
recall \eqref{eq:zeta} and, noting that we are restricting ourselves to $\mathbb{R}^2$, 
the inverse estimate 
\begin{equation}
|\nabla v_h |_{0,p} \leq Ch^{-1} | v_h |_{0,p}
\qquad v_h \in S_h,~~\mbox{for}~p\in[1,\infty].
\label{inverse}
\end{equation}
Summing over $n$, multiplying by $\Delta t$ and recalling that $\Delta t\leq Ch^2$ yields, in light of the 
bounds from Lemmas~\ref{lemma_phase_digm_4_1_varphi} and \ref{lemma_phase_digm_4_1_u}, that
\begin{align*}
\timeint{| \delta(\vphi_h) \nabla \uhl |_0^2}
& \leq 
C\timeint{\lnormwH{\nabla \ulinehplus}{\zeta(\vphi_h^{-})}^2} + Ch^{-2}\timeint{ \lnormH{\ulinehplus - \ulinehminus}^2} \nonumber \\
& \quad + Ch^{-2}\timeint{ \lnormH{\vphihplus - \vphihminus}^2 } 
\leq  C + C h^{-2} \Delta t \leq C.
\end{align*}

In what follows, we fix $p > 2$, and bound
\begin{equation}
\label{eq_delta_h_u_ht_bound}
	\timeint{\| \partial_t(\delta(\vphi_h) \uhl) \|_{(W^{1,p}(\Omega))'}^2 }.
\end{equation}
Let $\psi \in W^{1,p}(\Omega)$ be arbitrary. Then, noting Lemma \ref{lemma_phase_max_on_u_h}, we have
\begin{align}
(\partial_t (\delta(\vphi_h) \uhl), &\psi)_{((W^{1,p})', W^{1,p} )} = 
(\partial_t (\delta(\vphi_h) \uhl), \psi) 
\leq C \left| (\varphi_h \partial_t\varphi_{h} \uhl, \psi) \right| 
+ C \left| (\delta(\vphi_h) \uhtl, \psi) \right| \nonumber \\
&\leq C  \lnormH{\psi} \lnormH{\partial_t\varphi_{h}} 
+ \| \uhtl \|_{(H^1(\Omega))' } \| \psi \delta(\vphi_h) \|_{1} \nonumber \\
 &\leq C  \lnormH{\psi} \lnormH{\partial_t\varphi_{h}} 
 + \| \uhtl \|_{(H^1(\Omega))' } \left( \| \psi \|_{1} 
+ \| \psi \|_{0,\infty} \| \delta(\vphi_h) \|_{1} \right) \nonumber \\
 &\leq C  \| \psi \|_{1,p} \left( \lnormH{\partial_t\varphi_{h}} + \| \uhtl \|_{(H^1(\Omega))' }(1+ \| \vphi_h \|_{1} )\right). \nonumber 
\end{align}
Using bounds from Lemmas \ref{lemma_phase_digm_4_1_varphi} and \ref{lemma_phase_digm_4_2_u}, and the continuous embedding $W^{1,p} (\Omega) \hookrightarrow C^0(\overline{\Omega})$, we deduce a bound on (\ref{eq_delta_h_u_ht_bound}). This, combined with the bound $\lnormH{\nabla (\delta(\vphi_h) \uhl)}^2 \leq C$ 
and a well-known compactness results 
yields the desired strong convergence result \eqref{eq:lem36a}.
The proof for $\zeta(\varphi) u$ and \eqref{eq:lem36b} is analogous.
\end{proof}

We define the set
\[
	U := \Set{ (x,t) \in \Omega \times (0,T) | \zeta(\vphi(x, t)) > 0 } .
\]
As $\ul\in L^\infty (\Omega \times (0,T))$,  
the regularity theory for parabolic variational inequalities, see \cite{parabolic_regularity_theory_digm}, gives 
\[
 	\vphi \in L^p(0,T; W^{2,p}(\Omega)), \quad \mbox{and} \quad \partial_t\varphi  \in L^p(0, T; L^p (\Omega)), \quad \mbox{ for } 1 \leq p < \infty .
\]
Thus, by the Aubin--Lions lemma, 
$\vphi \in C^0 (\overline{\Omega \times (0, T)})$, and hence the set $U$ is open.

\begin{Lemma}
\label{lemma_phase_digm_4_4_zeta}
The function $F$ in \eqref{digm_4_3_conv_2} satisfies $F = \chi (U) \zeta(\vphi) \nabla \ul$ almost everywhere in $\Omega \times (0,T)$, where $\chi (U)$ is the characteristic function of $U$.
\end{Lemma}
\begin{proof}
In order to identify $F$ on $U$ we show that $(\zeta(\vphi)^2 \ul)_{x_i} \in L^2(0, T; L^2(\Omega))$ and 
\begin{equation}
\label{equation_in_proof_lemma_phase_digm_4_4_zeta}
(\zeta(\vphi)^2 \ul)_{x_i} = \zeta(\vphi) \ul \varphi_{x_i} + F_i \zeta(\vphi), \ i= 1,2 .
\end{equation}
By Lemma~\ref{lemma_phase_digm_4_3_deltazeta} we have $\zeta(\vphi_h) \uhl \rightarrow \zeta(\vphi) \ul$ in $L^2(0, T; L^2(\Omega))$, so that 
\[
\timeint{ (\zeta(\vphi)^2 \ul, \psi_{x_i})} 
= \lim_{h \rightarrow 0} \timeint{ (\zeta(\vphi_h)^2 \uhl, \psi_{x_i})} , 
\quad \forall \psi \in C^\infty_0(\Omega \times (0,T)).
\]
Using integration by parts on the right hand side integral we have
\begin{align}
\timeint{ (\zeta(\vphi_h)^2 \uhl, \psi_{x_i})} 
&=
- \timeint{ (\zeta(\vphi_h)^2 u_{h,x_i}, \psi) } 
- \timeint{ ((\zeta(\vphi_h)^2)_{x_i} \uhl, \psi )} \nonumber \\
&= - \timeint{(\zeta(\vphi_h) (\zeta(\vphi_h) u_{h,x_i}), \psi )} 
- \timeint{(\zeta(\vphi_h) \uhl \varphi_{h, x_i}, \psi )} . \nonumber 
\end{align}
Since $\zeta(\vphi_h) \uhl \rightarrow \zeta(\vphi) \ul$ in $L^2(0,T; L^2(\Omega))$ (by Lemma~\ref{lemma_phase_digm_4_3_deltazeta}), the dominated convergence theorem implies that $\zeta(\vphi_h) \uhl \psi \rightarrow \zeta(\vphi) \ul \psi$ in $L^2(0,T; L^2(\Omega))$. Using this, and $\nabla \varphi_h \rightharpoonup \nabla \varphi$ in $L^2(0,T; [L^2(\Omega)]^2)$ from (\ref{eq_digm_all_convergences_1}), we have
\[
\timeint{ (\zeta(\vphi_h) \uhl \varphi_{h, x_i}, \psi )} \rightarrow 
\timeint{ ( \zeta(\vphi) \ul \varphi_{ x_i}, \psi )}  , 
\quad \forall \psi \in C^\infty_0(\Omega \times (0,T)).
\]
By (\ref{digm_4_3_conv_2}) and (\ref{equation_phase_digm_convergence_of_weights_zeta}) we have
\[
\timeint{ (\zeta(\vphi_h) (\zeta(\vphi_h) u_{h,x_i}), \psi )} \rightarrow 
\timeint{ (\zeta(\vphi) F_i, \psi )} , 
\quad \forall \psi \in C^\infty_0(\Omega \times (0,T)).
\] 
Thus, as $\psi$ is arbitrary, we have (\ref{equation_in_proof_lemma_phase_digm_4_4_zeta}) almost everywhere.

We now identify $F$ on $U$. Let $\psi \in C^\infty_0(U)$ be arbitrary. Using integration by parts we have
\begin{align*}
- \int_U \ul_{x_i} \psi \, {\rm d}x\,  {\rm d}t
= \int_U \ul \psi_{x_i} \, {\rm d}x\,  {\rm d}t
& = \int_U \zeta(\vphi)^2 \ul \frac{1}{\zeta(\vphi)^2} \psi_{x_i} 
\, {\rm d}x\,  {\rm d}t
\nonumber \\
&= - \int_U (\zeta(\vphi)^2 \ul)_{x_i} \frac{\psi}{\zeta(\vphi)^2} 
\, {\rm d}x\,  {\rm d}t
+ \int_U \frac{1}{\zeta(\vphi)} \ul \varphi_{x_i} \psi \, {\rm d}x\,  {\rm d}t.
\end{align*}
Substituting in (\ref{equation_in_proof_lemma_phase_digm_4_4_zeta}), we have
\[
- \int_U u_{x_i} \psi \, {\rm d}x\,  {\rm d}t
= - \int_U F_i \frac{1}{\zeta(\vphi)} \psi \, {\rm d}x\,  {\rm d}t.
\]
Since $\psi$ is arbitrary, this gives us that $u_{x_i} \zeta(\vphi) = F_i$ almost everywhere on $U$.

It remains to identify $F$ on $U^c := \Omega \times(0,T) \setminus U$. 
Let  $\psi \in C^\infty_0(\Omega \times (0,T))$ be arbitrary. We use that $1 - \chi (U) = 0$ on $U$ to give
\begin{align}
& \left| \int_{U^c}  \zeta(\vphi_h^-) u_{h,x_i}^{+} \psi \, {\rm d}x\,  {\rm d}t
\right|  
= \left| \timeint{ \phaseinth{ \zeta(\vphi_h^-) u_{h,x_i}^{+} (1-\chi(U)) \psi }} \right|  \nonumber \\
 & \leq  \left( \timeint{\phaseinth{ \zeta(\vphi_h^-) | \nabla u^+_h |^2 }} \right) ^{\frac{1}{2}} \left( \timeint{\phaseinth{ \zeta(\vphi_h^-) (1- \chi (U))^2 \psi^2 }} \right) ^{\frac{1}{2}},
\nonumber \\
 & \leq C \left( \timeint{\phaseinth{ \zeta(\vphi_h^-) (1- \chi (U))^2 \psi^2 }} \right)^{\frac{1}{2}},
\label{eq_digm_0_6_68}
\end{align}
where we have recalled (\ref{equation_digm_4_1_L2_bounds_u}).
By (\ref{equation_phase_digm_convergence_of_weights_zeta})
\[
	\left( \timeint{\phaseinth{ \zeta(\vphi_h^-) (1- \chi (U))^2 \psi^2 }} \right) ^{\frac{1}{2}} \rightarrow 0 .
\]
Thus, from (\ref{eq_digm_0_6_68}), we have
\[
\int_{U^c}  \zeta(\vphi_h^-) u_{h,x_i}^{+} \psi \, {\rm d}x\,  {\rm d}t 
\rightarrow 0.
\]
Recalling (\ref{digm_4_3_conv_2}), we have
\[
\timeint{ \phaseinth{ \zeta(\vphi_h^-) u_{h,x_i}^{+} \psi }} \rightarrow \timeint{\phaseinth{ F_i \psi }}.
\]
Thus we conclude that $F_i = 0$ almost everywhere in $U^c$, $i=1,2$.
\end{proof}

\begin{Theorem}
\label{thm_digm_theorem1_u}
The functions $\varphi$ and $u$ in \eqref{eq:conv} satisfy 
\begin{align}
\label{equation_digm_theorem1_u}
& \epsilon^2 \timeint{\psi (\utl, \eta)_{((H^1)', H^1)}} 
+ \timeint{ \psi \int_{\Set{\zeta > 0}}  \zeta(\vphi) \nabla \ul \cdot \nabla \eta \; {\rm{d}}x}
+ \frac{1}{\epsilon\alpha} \timeint{ \psi (\delta(\vphi) \ul, \eta ) }
 \nonumber \\ & \qquad
~~~~~~~~~~~~~~~~~~~~~~~~~~= \frac{Q}{\epsilon} \timeint{\psi (\delta(\vphi) , \eta )}
- \timeint{\psi (\zeta(\vphi), \eta )} 
\end{align}
for an arbitrary $\eta\in H^1(\Omega)$ and $\psi \in C^\infty_0(0,T)$.
\end{Theorem}
\begin{proof}
Choosing arbitrary functions 
$\eta\in H^1(\Omega)$ and $\psi \in C^\infty_0(0,T)$, we know that there exists
a sequence $(\eta_h) \subset S_h$ such that $\eta_h \to \eta$ in $H^1(\Omega)$
as $h\to0$. 
Multiplying (\ref{equation_phase_FEM_ua}) by $\psi$, and integrating over $t$ gives
\begin{align}
\label{equation_digm_theorem1_u_h}
\underbrace{ \epsilon^2 \timeint{\psi \left( \uhtl ,\eta_h \right)_h}}_{(1)}
&+ \underbrace{ \timeint{\psi \left(\zeta(\vphi_h^-) \nabla \ulinehplus ,\nabla \eta_h \right)}}_{(2)} 
+ \underbrace{ \frac{1}{\epsilon\alpha} \timeint{ \psi \left( \delta(\vphi_h^-) \ulinehplus , \eta_h \right)_h }}_{(3)} \nonumber \\
&= \underbrace{ \frac{Q}{\epsilon} \timeint{\psi \left( \delta(\vphi_h^-), \eta_h \right)_h }}_{(4)} 
- \underbrace{ \timeint{\psi \left( \zeta(\vphi_h^-), \eta_h \right)_h }}_{(5)} . 
\end{align}
For all but $(2)$ we use the well known inequality (\ref{numerical_int}).
For $(1)$ we use (\ref{eq:uht}). 
For $(2)$ we use (\ref{digm_4_3_conv_2}) and Lemma~\ref{lemma_phase_digm_4_4_zeta}.
For $(3)$ we use \textcolor{black}{(\ref{equation}).} 
For $(4)$ we use (\ref{equation_phase_digm_convergence_of_weights_delta}). 
For $(5)$ we use (\ref{equation_phase_digm_convergence_of_weights_zeta}).
\end{proof}
\begin{Remark}
Due to the inverse inequality (\ref{inverse}) that is used in the proof of Lemma \ref{lemma_phase_digm_4_3_deltazeta}, our proof of the convergence result in Theorem \ref{thm_digm_theorem1_u} does not extend to $\mathbb{R}^3$. 
\end{Remark}

\setcounter{equation}{0}
\section{Numerical results}
\label{sec:nr}
In this section we display some computational simulations of tumour growth. 
In all computations we use the finite element approximation $\mathbb{P}_h$ that results from taking $\vartheta=1$ in \eqref{eq:strong}.
We use the finite element toolbox Alberta 2.0, \cite{ALBERTABook}, 
to implement our approximation \eqref{eq:fea}. 
In order to increase the efficiency of the computations, we employ 
the mesh refinement strategy presented in \cite{eks}, which gives rise to a fine mesh in the interfacial region where $|\varphi_h^n|<1$, a coarse mesh exterior to the tumour where $\varphi_{h}^{n} = -1$ and a standard sized mesh in the interior of the tumour where $\varphi_{h}^{n} = 1$, for more details on a similar mesh refinement strategy see \cite{Barrett2005}. 
The result of this refinement strategy can be seen in Figure \ref{f:meshes}, which displays enlarged sections of the computational meshes associated with the simulations presented in Figure \ref{f:a10b01} at $t=45$ (left) and Figure \ref{f:a01b01} at $t=7$ (right). 
We denote the maximum diameter of the triangles in the three meshes by $h_{max,f}=\max_{\sigma\in\mathcal T_f^n}h_\sigma, h_{max,m}=\max_{\sigma\in\mathcal T_m^n}h_\sigma$ and $h_{max,c}=\max_{\sigma\in\mathcal T_c^n}h_\sigma$, where
$$
\mathcal T^n_f:= \lbrace \sigma \in \mathcal T_h \, | \, |\varphi_h^n(x)|<0.99 ~\forall\,  x\in \sigma \rbrace, \quad\mathcal T^n_m:= \lbrace \sigma \in \mathcal T_h \, | \, \varphi_h^n(x)=1 ~\forall~\, x \in \sigma\rbrace, 
$$
and
$$
\mathcal T^n_c:= \lbrace \sigma \in \mathcal T_h \, | \, \varphi_h^n(x)=-1 ~\forall\, x\in \sigma \rbrace.
$$
In all our computations we fix 
$h_{\max,f} \leq h_{\max,m} \leq h_{\max,c}$ with
$h_{\max,m} / h_{\max,f} \leq 2^{4}$ and 
$h_{\max,c} / h_{\max,m} \leq 2^{7}$. 
\begin{figure}[htbp]
\begin{center}
\includegraphics[width=.35\textwidth,angle=0]{{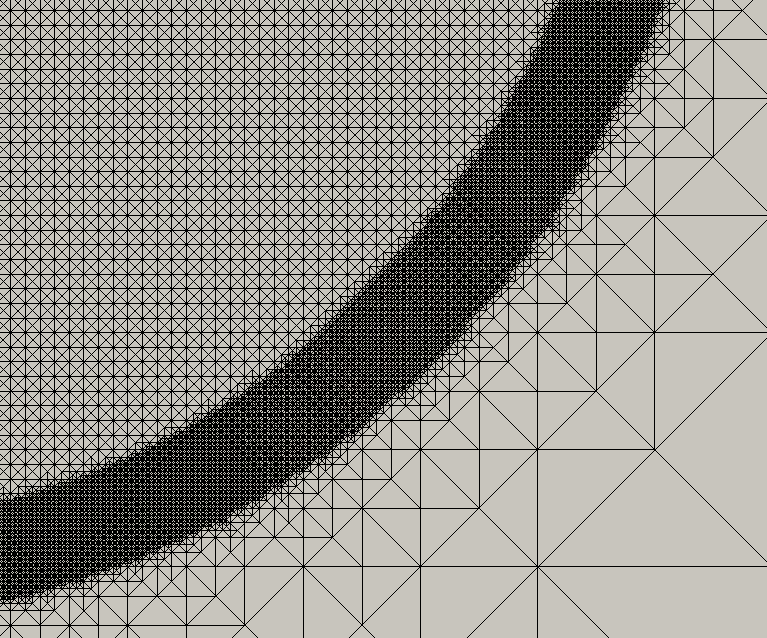}}\hspace{15mm}
\includegraphics[width=.35\textwidth,angle=0]{{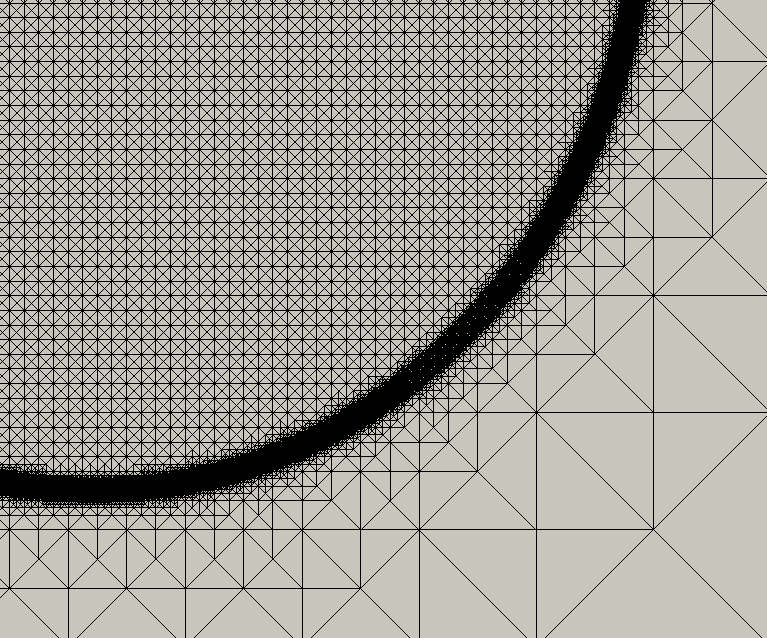}}
\end{center}
\caption{
Enlarged sections of the meshes associated with the results in Figure \ref{f:a10b01} at $t=45$ (left) and Figure \ref{f:a01b01} at $t=7$ (right).}
\label{f:meshes}
\end{figure}
The linear system resulting from (\ref{equation_phase_FEM_u}) was solved using 
a preconditioned conjugated gradient solver
with diagonal preconditioning, while a projected SOR method, 
see \cite{ellock}, was used to solve the system resulting from 
the variational inequality (\ref{equation_phase_FEM_v}).  The results are visualised in Paraview \cite{paraview_book}.\\[2mm]

Noting Lemma \ref{lem_exist}, in all simulations we take $\Delta t<\frac{\varepsilon^2}{\beta}$. In addition we set $h_{max,c}\approx 2.5$ and $h_{max,m}\approx0.02$ and we take the the initial pressure to be $u_0(x)= 0$ for all $x\in \Omega$. 
Unless otherwise stated we take the initial geometry $\Gamma(0)$ to be an ellipse with length $0.5$ and 
height $1$, that is we choose the initial profile $\varphi_0$ as
\begin{equation*} 
\varphi_0(x) = \begin{cases}
1 & r(x) \geq \frac{\epsilon\pi}2, \\
\sin\frac{r(x)}\epsilon &  -\frac{\epsilon\pi}2< r(x) < \frac{\epsilon\pi}2 ,\\
-1 & r(x) \leq -\frac{\epsilon\pi}2 ,\\
\end{cases}
\end{equation*}
where $r(x) = 1 - \sqrt{4 x_1^2 + x_2^2}$.
The values of $Q,~\alpha,~\beta,~\epsilon,~\Delta t$ and $h_{max,f}$ used
in all our simulations are stated in Table \ref{t:param}. 
Due to the symmetry of the problems we consider, in all simulations we only solve in the positive quadrant and apply homogeneous Neumann boundary conditions on the boundaries $x_i=0$, $i=1,2$. 

\begin{table}[!h]\begin{center}
 \begin{tabular}{ ccccccc }
~ & ~~$Q$~~ & ~~$\alpha$~~& $\beta$ &$\epsilon$ & $\Delta t$ & $h_{max,f}$ \\
 \hline
  \hline
Figure \ref{f:a10b01} & $1$ & $1$ & $0.1$& $0.04$ & $10^{-3}$ &0.0048\\
Figure \ref{f:a01b01} & $1$ & $0.1$& $0.1$& $0.01$ & $10^{-4}$ & 0.0024\\
Figure \ref{f:qvary} & $\{0.5,0.75,1.5\}$ & $1$ & $0.1$& $0.04$ & $10^{-3}$ &0.0048\\
Figure \ref{f:avary} & $1$  & $\{0.5,2,7\}$ & $0.1$& $0.04$ & $10^{-3}$ &0.0048\\
Figure \ref{f:bvary} & $1$  & $1$ & $\{0.1,0.2,0.5\}$ & $0.04$ & $10^{-3}$ &0.0048\\
\end{tabular}
\end{center}
\caption{Parameters values for Figures \ref{f:a10b01}  -- \ref{f:bvary}. }
\label{t:param}
\end{table}

\subsection{Radially symmetric solutions}\label{rad}
In the first set of experiments we investigate the accuracy of the numerical scheme as: (i) $\Delta t$ varies, (ii) $\varepsilon$ is reduced, (iii) $\beta$ is increased. To this end, we consider radially symmetric solutions by setting $\Gamma(t) \subset \mathbb{R}^2$ to be a circle with radius $R(t)$, and we express $u$ in polar coordinates such that $u(r, \theta)=u(r)$. In this setting 
(\ref{eq:eks}) reduces to 
\begin{equation}
	u(r) = \frac{1}{4} r^2 + \alpha Q  - \frac{\alpha}{2} R - \frac{1}{4} R^2,
\label{ueqn}
\end{equation}
\begin{equation}
	R'(t) = - \frac{\beta}{R} + \frac{1}{\alpha} u(r) =  -\frac{\beta}{R}+Q -\frac{R}{2}, ~~~R(0) := R_0.
	\label{reqn}
\end{equation}
For each of the simulations we compute the following error 
$$
\mathcal{E}_r := \sum_{n=0}^{50}|R_{h}(t_n)-R(t_n)|^2,
$$
which is the error between the radii, $R_{h}(t_n)$, for $t_n=0.01n$, obtained from the finite element approximation  (\ref{equation_phase_FEM_u}), (\ref{equation_phase_FEM_v}) and the radii, $R(t_n)$, of the corresponding  analytical solution, solved using MATLAB's standard solver for ordinary differential equations, ode45. \\[2mm]
We first investigate an appropriate choice of $\Delta t$, relative to  $h_{max,f}$. To this end, we set $\alpha=Q = R_0=1.0$, $\beta=0.1$ then, for $\varepsilon=0.02$ and $\varepsilon=0.01$, we consider two values of $h_{max,f}$ and four values of $\Delta t$.  \begin{table}[!h]\begin{center}
 \begin{tabular}{ c|cccc }
 &$\Delta t=0.5h_{max,f}$&$\Delta t=0.2h_{max,f}$&$\Delta t=0.1h_{max,f}$&$\Delta t=0.05h_{max,f}$\\
  \hline
$h_{max,f}\approx0.25\varepsilon$&0.0101793 &0.0106555 &0.0107909 &0.0139744 \\
$h_{max,f}\approx0.125\varepsilon$&0.00339958 &0.00225313 &0.00200782 &0.00435587
\end{tabular}
\end{center}
\caption{$\mathcal{E}_r$:  $\varepsilon=0.02$ with $\alpha=Q = R_0=1.0$, $\beta=0.1$, $\Delta t = 0.2h_{max,f}$}
\label{t:dt}
\end{table}
\begin{table}[!h]\begin{center}
 \begin{tabular}{ c|cccc }
 &$\Delta t=0.5h_{max,f}$&$\Delta t=0.2h_{max,f}$&$\Delta t=0.1h_{max,f}$&$\Delta t=0.05h_{max,f}$\\
  \hline
$h_{max,f}\approx 0.25\varepsilon$&0.0252947 &0.0201578&0.0184133 &0.0181777 \\
$h_{max,f}\approx 0.125\varepsilon$&0.00494710 &0.00438694 &0.00434753 &0.00494933 
\end{tabular}
\end{center}
\caption{$\mathcal{E}_r$: $\varepsilon=0.01$ with $\alpha=Q = R_0=1.0$, $\beta=0.1$, $\Delta t = 0.2h_{max,f}$}
\label{t:dt2}
\end{table}
The resulting errors are displayed in Table \ref{t:dt} (for $\varepsilon=0.02$) and Table \ref{t:dt2} (for $\varepsilon=0.01$). From these tables we see that  there is very little 
difference in the errors for the four values of $\Delta t$, thus for the remaining radially symmetric computations, unless otherwise specified, we set $\Delta t=0.2h_{max,f}$.
\begin{Remark}
The convergence result in Section \ref{sec:convergence} relies on the requirement that $\Delta t\leq Ch^2$, 
however from Tables \ref{t:dt} and \ref{t:dt2} we see that it is sufficient to take $\Delta t\leq Ch$ in practice. 
\end{Remark}
\begin{table}[!h]\begin{center}
 \begin{tabular}{ c|cccc }
$~$ &$\varepsilon=0.04$&$\varepsilon=0.02$&$\varepsilon=0.01$&$\varepsilon=0.005$\\
  \hline
$ \varepsilon \approx 4h_{max,f}$ &0.00929059& 0.0106555&0.0201578& 0.0304609\\
$ \varepsilon\approx 8h_{max,f}$ &0.0181973&0.00225313&0.00438694&0.008757458
\end{tabular}
\end{center}
\caption{$\mathcal{E}_r$:  $\alpha=Q = R_0=1.0$, $\beta=0.1$, $\Delta t = 0.2h_{max,f}$}
\label{t:epsr}
\end{table}
Next we consider the accuracy of the scheme as $\varepsilon$, and subsequently $h_{max,f}$ and $\Delta t=0.2h_{max,f}$, are reduced. To this end we set $\alpha=Q = R_0=1.0$, $\beta=0.1$ and consider four values of $\varepsilon$. For each value of $\varepsilon$ we consider two values of $h_{max,f}$.  
The errors $\mathcal{E}_r$ are displayed in  Table \ref{t:epsr}, from which we see that taking $\varepsilon\approx 4h_{max,f}$ yields, for all values of $\varepsilon$, errors that are of the same order of magnitude.  Whereas  taking $\varepsilon\approx 8h_{max,f}$ yields errors that are the same order of magnitude for $\varepsilon=0.02$ and $\varepsilon=0.01$, but gives larger errors for $\varepsilon=0.04$ and $\varepsilon=0.005$.\\[2mm]
We conclude our radially symmetric results by considering the accuracy of the scheme when $\beta$ is increased from $\beta=0.1$ to $\beta=1$. 
In these computations we set $Q = R_0=1.5$, $\alpha=1$ and consider two values of $\varepsilon$. We note that increasing $Q$ and $R_0$ from $1$ to $1.5$ ensures growth of the initial circle, recall (\ref{reqn}). 
For each value of $\varepsilon$ we consider three values of $h_{max,f}$. In all computations we set $\Delta t= 0.05h_{max,f}$, since this results in the restriction $\Delta t<\frac{\varepsilon^2}{\beta}$, recall Lemma \ref{lem_exist}, always being satisfied. 
By comparing Table \ref{t:betar}  with Table \ref{t:epsr} we see that for $\beta=1$, even when the value of $h_{max,f}$ is half the size of the values used when $\beta=0.1$, the errors for $\beta=1$ are larger than the errors for $\beta=0.1$.
We infer from Tables \ref{t:epsr} and \ref{t:betar}  that in order to maintain the accuracy of the scheme, a relationship of the form $\varepsilon = \Lambda h_{max,f}$ should be used, with $\Lambda$ increasing as $\beta$ increases. 
\begin{table}[!h]\begin{center}
 \begin{tabular}{ c|cc }
$~$ &$\varepsilon=0.04$&$\varepsilon=0.02$\\
  \hline
$ \varepsilon \approx 4h_{max,f}$ &0.0891700& 0.0844617 \\
$ \varepsilon \approx 8h_{max,f}$ &0.0468858& 0.0253982 \\
$ \varepsilon\approx 16h_{max,f}$ &0.0350065&0.0164192 
\end{tabular}
\end{center}
\caption{$\mathcal{E}_r$:  $\beta=1$,  $Q = R_0=1.5$, $\alpha=1$, $\Delta t = 0.05h_{max,f}$}
\label{t:betar}
\end{table}

\subsection{Comparison with results in \cite{eks}}
In the second set of experiments, Figures \ref{f:a10b01} and \ref{f:a01b01}, we investigate the effect
of using $\vartheta=1$ in \eqref{eq:strong}, which we included for the
numerical analysis in this paper. We recall that a discretisation of
\eqref{eq:strong} with $\theta=0$ was considered in \cite{eks}, and so we
compare our numerical results with those obtained in that paper.
In particular, 
in Figures \ref{f:a10b01} and \ref{f:a01b01} we display results that correspond to the ones presented respectively in Figures 10 and 12 of \cite{eks}. 
For both simulations we set $\Omega=(-5,5)^2$, $Q=1$ and $\beta=0.1$. 
The values of $\alpha,~\epsilon,~\Delta t$ and $h_{max,f}$, which differ for the two simulations, are stated in Table \ref{t:param}. 
For each figure we display plots of the solution $\varphi_h^n$ in the top row and the pressure $u_h^n$, restricted to the region in which $\varphi_h^n>-1$, i.e. the interior of the tumour, in the bottom row. To be consistent with Figures 10 and 12 in \cite{eks}, the plots in Figure \ref{f:a10b01} are displayed at $t=0,30,45$ and the plots in Figure \ref{f:a01b01} are displayed at $t=0,3,7$. 
For both sets of results we see a close resemblance to the results in \cite{eks}, with the geometries being visually indistinguishable and only small differences in the values of the discrete pressures.
Overall we are satisfied that choosing $\vartheta=1$ in \eqref{eq:strong},
compared to $\vartheta=0$, only has a negligible effect on the numerical
results for the evolutions we are interested in.

 \subsection{The influence of $Q$, $\alpha$ and $\beta$}
In the remaining figures, Figures \ref{f:qvary} -- \ref{f:bvary}, we present some parameter studies that investigate 
the influence of the parameters $Q$, $\alpha$ and $\beta$, respectively. 
In all the simulations we set $\Omega=(-10,10)^2$ and we display the pressure $u_h^n$ in the region in which $\varphi_h^n>-1$. Associated with this parameter study we note that in Section 3.1.2 of \cite{eks} a linear stability analysis on a more complicated version of the model (\ref{eq:eks}), that includes an additional curvature term on the left hand side of (\ref{eq_f2_uOnGamma}), yields steady state circular solutions if $3\alpha\beta > 2Q^3$. 

Figure \ref{f:qvary} displays $u_h^n$ obtained by setting $\alpha=1$, $\beta=0.1$ and $Q=0.5$ (top row), $Q=0.75$ (middle row) and $Q=1.5$ (bottom row). From this figure we see that for $Q=0.5$ the ellipse evolves to form the expected steady state circular solution, since $3\alpha\beta = 0.3 > 0.25 = 2Q^3$. For $Q=0.75$ the ellipse extends in the $x_2$ direction to produce an elongated geometry that is rounded at both ends. This geometry is reminiscent of the simulation displayed in Figure 8 of \cite{eks} that relates to the thin film evolution analysed in the Appendix of \cite{eks}. When $Q=1.5$ the ellipse evolves into a more complex geometry, with a rounded and compact structure.

From Figures \ref{f:a10b01} and \ref{f:a01b01} we saw the effect that varying $\alpha$ has on the solution, with $\alpha=1$ giving rise to a geometry with four distinct branches, while $\alpha=0.1$ leads to more rounded and compact structure. We investigate the effect of varying $\alpha$ further in Figure \ref{f:avary} in which we set $Q=1$ and $\beta=0.1$. We display $\alpha=0.5$ (top row), $\alpha=2$ (middle row) and $\alpha=7$ (bottom row). From this figure we see that $\alpha=0.5$ gives rise to a rounded and compact structure similar to the one seen in the bottom row of Figure \ref{f:qvary}, while for $\alpha=2$ four branches are present at $t=125$ and by $t=170$ these have split and evolved into eight branches. Taking $\alpha=7$ yields the expected steady state circular solution, since we have $3\alpha\beta = 2.1 > 2 = 2Q^3$. 

Finally, in Figure \ref{f:bvary} we see the effect that varying $\beta$ has on the solution. We set  $Q=\alpha=1$ and display $\beta=0.5$ (top row) and $\beta=0.2$ (middle row), together with $\beta=0.1$ (bottom row). Here we display $\beta=0.1$ solely for comparison purposes, since it is the same as the solution from Figure \ref{f:a10b01} except that here $\Omega=(-10,10)^2$ rather than $\Omega=(-5,5)^2$. 
In this figure we see that when $\beta=0.5$ the curvature term dominates the motion and the ellipse evolves to a shrinking circle that has almost disappeared by $t=2$. When $\beta=0.2$ an elongated geometry with rounded ends is observed that is similar to the one displayed in the middle row of Figure \ref{f:qvary}.

\begin{figure}[htbp]
\begin{center}
\subfloat[t=0]{\includegraphics[height=.15\textheight,angle=0]{{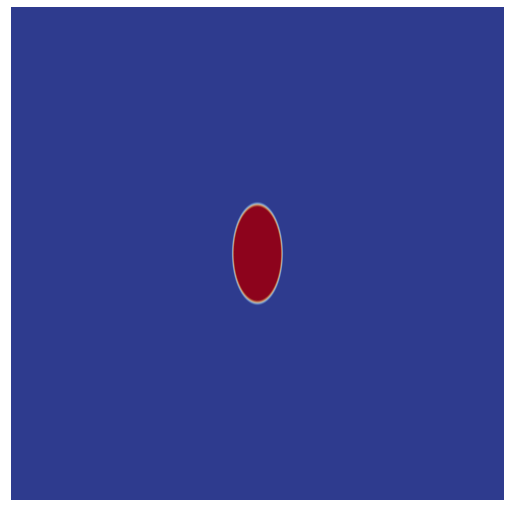}}}\hspace{0mm}
\subfloat[t=30]{\includegraphics[height=.15\textheight,angle=0]{{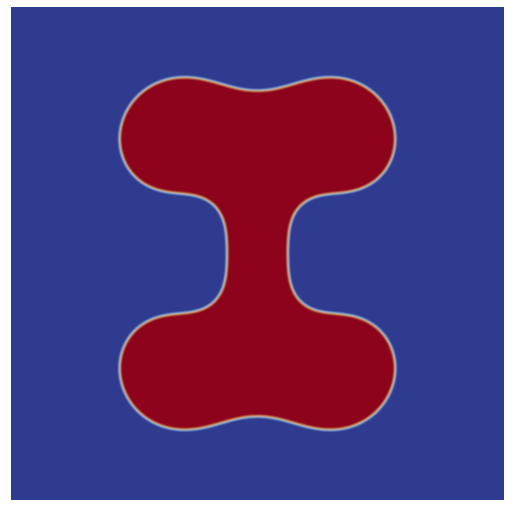}}}\hspace{0mm}
\subfloat[t=45]{\includegraphics[height=.15\textheight,angle=0]{{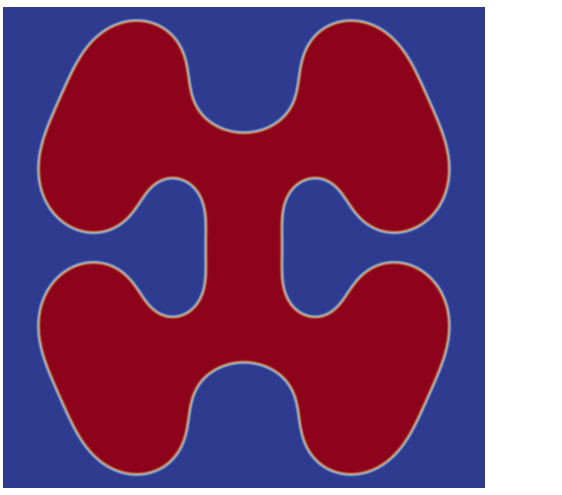}}}\hspace{0mm}
\end{center}
\begin{center}
\subfloat[t=0]{\includegraphics[height=.15\textheight,angle=0]{{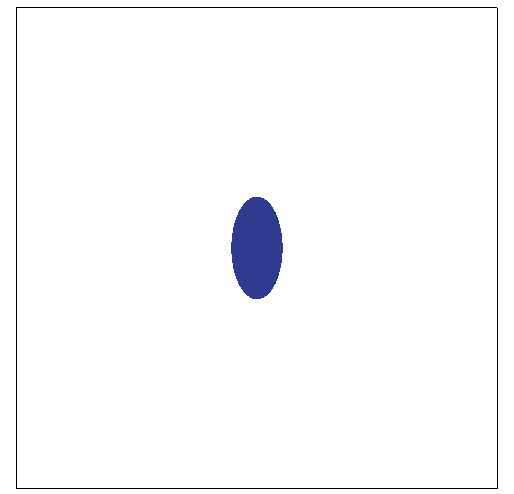}}}\hspace{0mm}
\subfloat[t=30]{\includegraphics[height=.15\textheight,angle=0]{{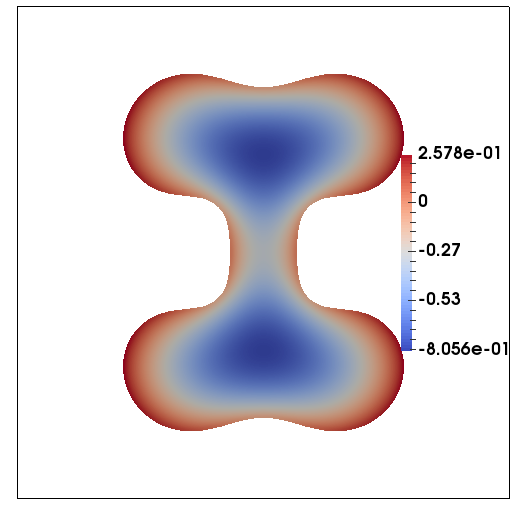}}}\hspace{0mm}
\subfloat[t=45]{\includegraphics[height=.15\textheight,angle=0]{{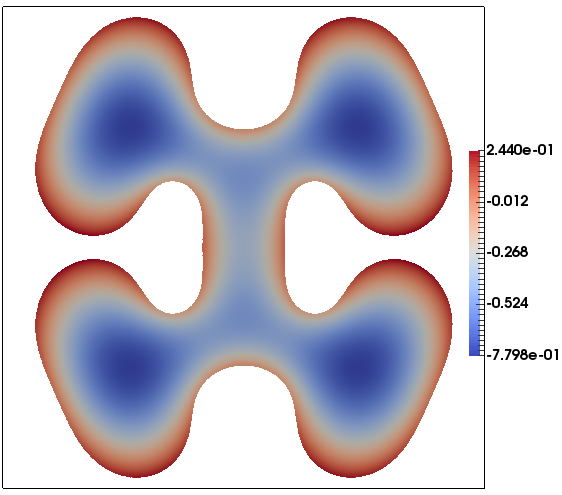}}}\hspace{0mm}
\end{center}
\caption{Simulations with $\Omega=(-5,5)^2$, $\alpha = 1$, $\beta = 0.1$ and $Q=1$: $\varphi_h$ (upper row) and $u_h$ (lower row).}
\label{f:a10b01}
\end{figure}

\begin{figure}[htbp]
\begin{center}
\subfloat[t=0]{\includegraphics[height=.15\textheight,angle=0]{{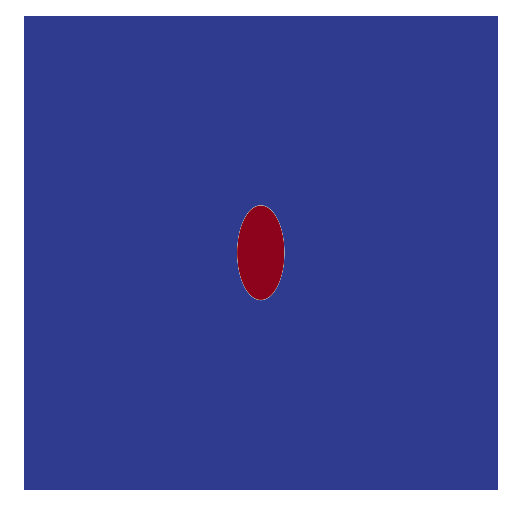}}}\hspace{0mm}
\subfloat[t=3]{\includegraphics[height=.15\textheight,angle=0]{{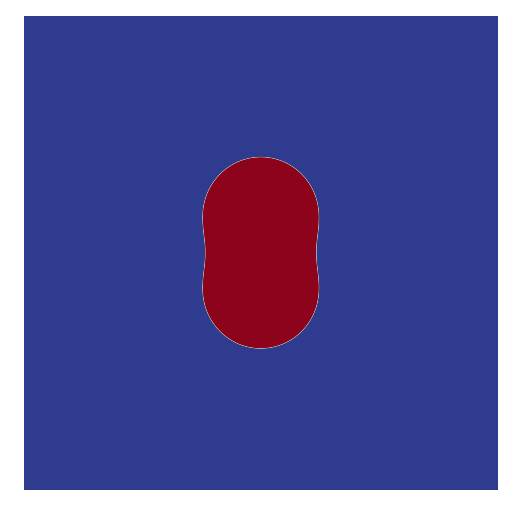}}}\hspace{0mm}
\subfloat[t=7]{\includegraphics[height=.15\textheight,angle=0]{{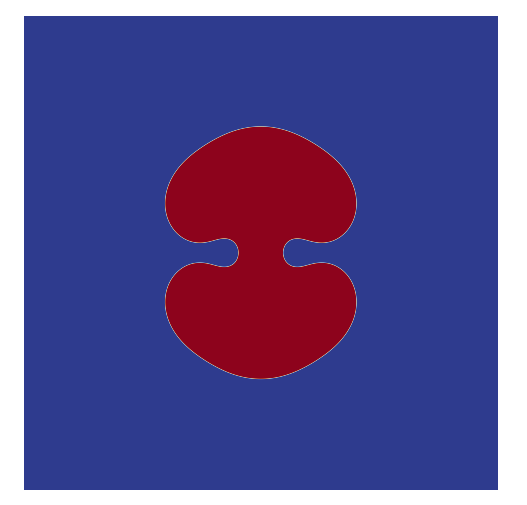}}}\hspace{0mm}
\end{center}
\begin{center}
\subfloat[t=0]{\includegraphics[height=.15\textheight,angle=0]{{pressure_id.png}}}\hspace{0mm}
\subfloat[t=3]{\includegraphics[height=.15\textheight,angle=0]{{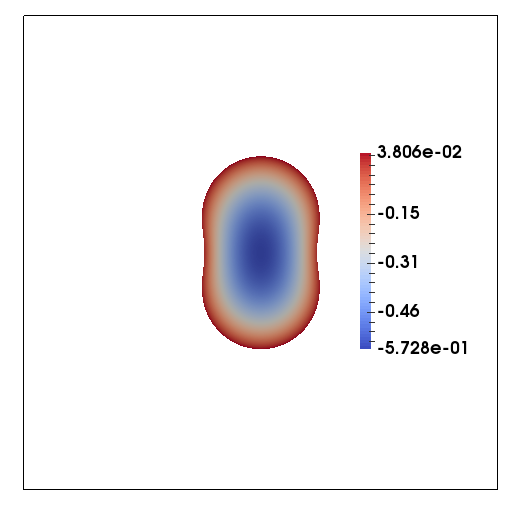}}}\hspace{0mm}
\subfloat[t=7]{\includegraphics[height=.15\textheight,angle=0]{{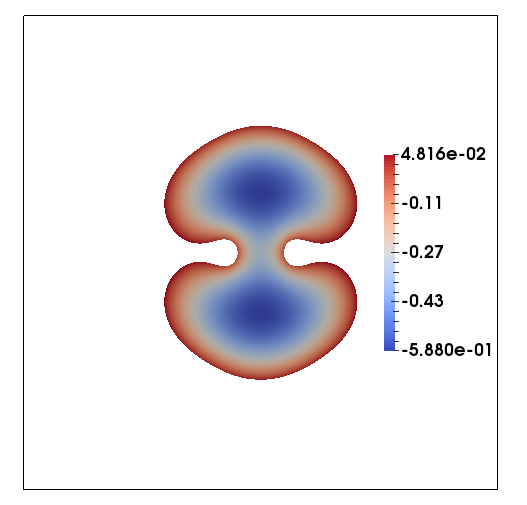}}}\hspace{0mm}
\end{center}
\caption{Simulations with $\Omega=(-5,5)^2$, $\alpha = 0.1$, $\beta = 0.1$ and $Q=1$: $\varphi_h$ (upper row) and $u_h$ (lower row).}
\label{f:a01b01}
\end{figure}

\begin{figure}[htbp]
\centering
\subfloat[t=0]{\includegraphics[height=.22\textheight,angle=0]{{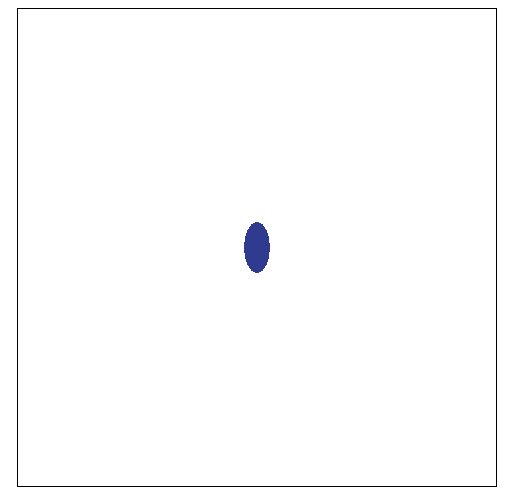}}}\hspace{-1mm}
\subfloat[t=2]{\includegraphics[height=.22\textheight,angle=0]{{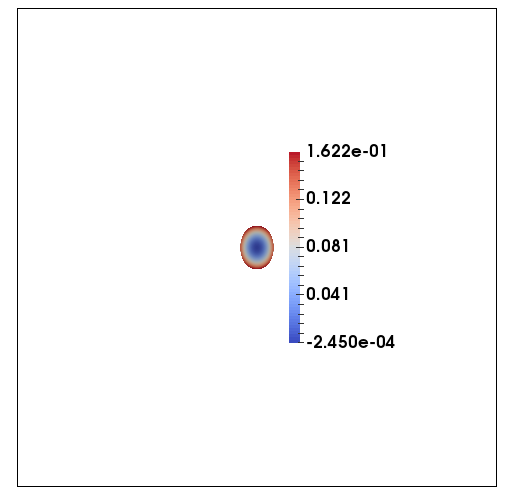}}}\hspace{-1mm}
\subfloat[t=25]{\includegraphics[height=.22\textheight,angle=0]{{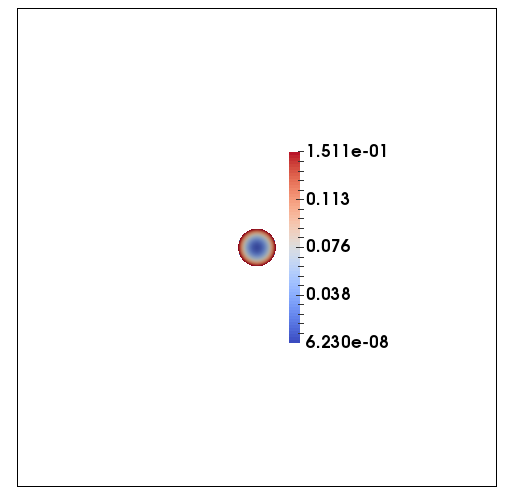}}}\hspace{0mm}
\subfloat[t=0]{\includegraphics[height=.22\textheight,angle=0]{{pressure_id_big.png}}}\hspace{-1mm}
\subfloat[t=140]{\includegraphics[height=.22\textheight,angle=0]{{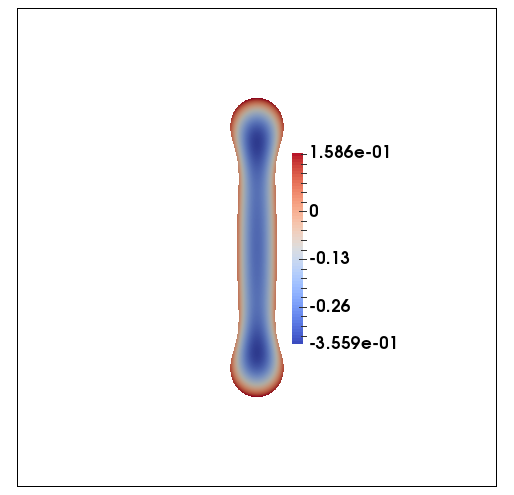}}}\hspace{-1mm}
\subfloat[t=230]{\includegraphics[height=.22\textheight,angle=0]{{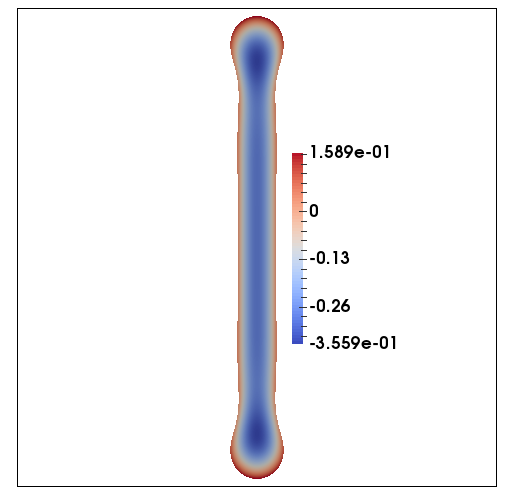}}}\hspace{0mm}
\subfloat[t=0]{\includegraphics[height=.22\textheight,angle=0]{{pressure_id_big.png}}}\hspace{-1mm}
\subfloat[t=7.5]{\includegraphics[height=.22\textheight,angle=0]{{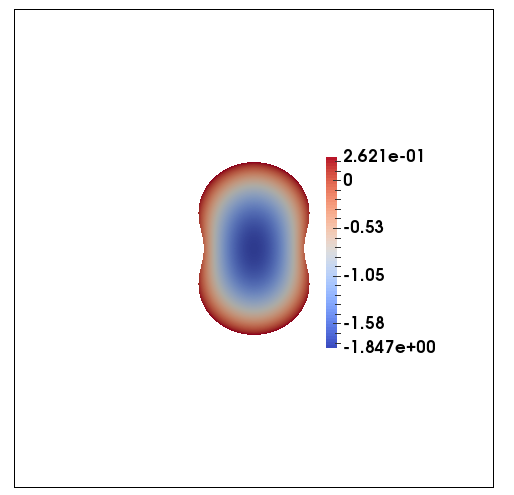}}}\hspace{-1mm}
\subfloat[t=13]{\includegraphics[height=.22\textheight,angle=0]{{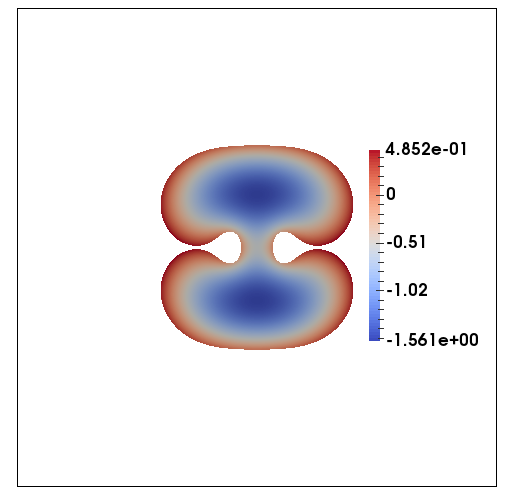}}}\hspace{0mm}
\caption{Simulations of $u_h$ with $\Omega=(-10,10)^2$, $\alpha = 1$, $\beta = 0.1$ and $Q=0.5$ (top row), $Q=0.75$ (middle row), $Q=1.5$ (bottom row).}
\label{f:qvary}
\end{figure}

\begin{figure}[htbp]
\centering
\subfloat[t=0]{\includegraphics[height=.22\textheight,angle=0]{{pressure_id_big.png}}}\hspace{-1mm}
\subfloat[t=12.5]{\includegraphics[height=.22\textheight,angle=0]{{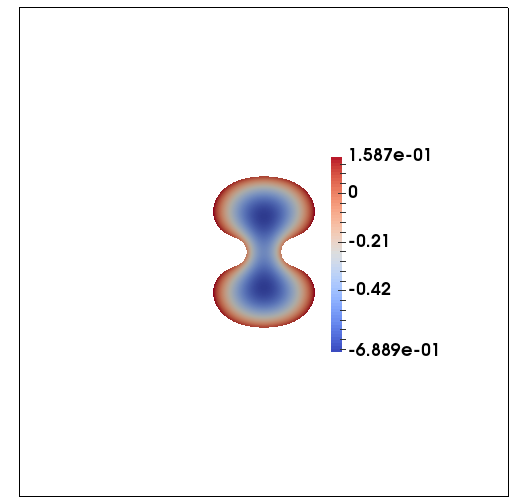}}}\hspace{-1mm}
\subfloat[t=19.5]{\includegraphics[height=.22\textheight,angle=0]{{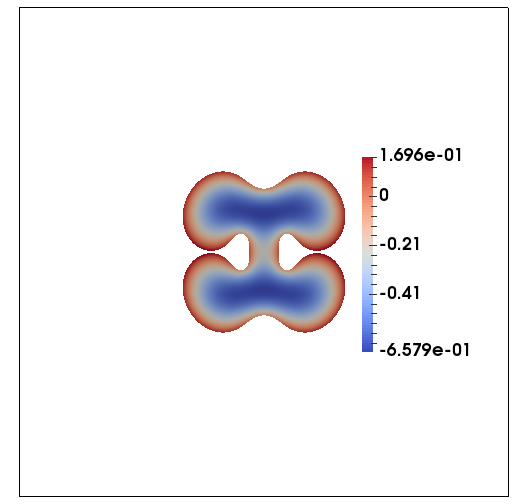}}}\hspace{0mm}
\subfloat[t=0]{\includegraphics[height=.22\textheight,angle=0]{{pressure_id_big.png}}}\hspace{-1mm}
\subfloat[t=125]{\includegraphics[height=.22\textheight,angle=0]{{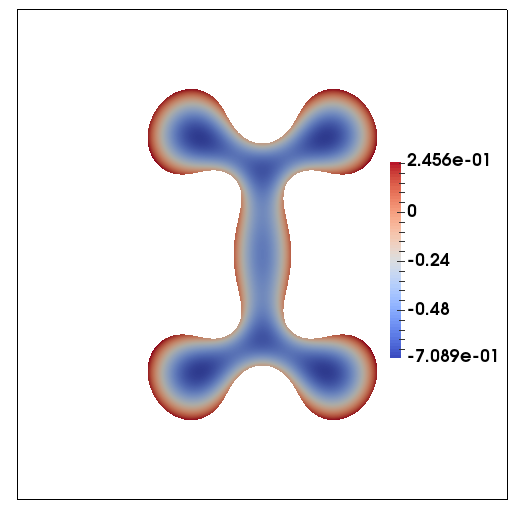}}}\hspace{-1mm}
\subfloat[t=170]{\includegraphics[height=.22\textheight,angle=0]{{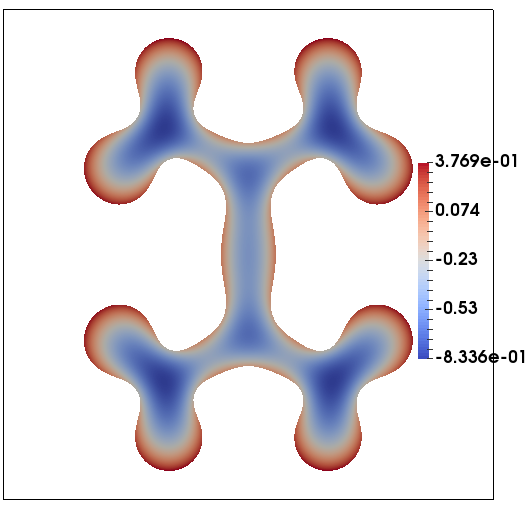}}}\hspace{0mm}
\subfloat[t=0]{\includegraphics[height=.22\textheight,angle=0]{{pressure_id_big.png}}}\hspace{-1mm}
\subfloat[t=2.5]{\includegraphics[height=.22\textheight,angle=0]{{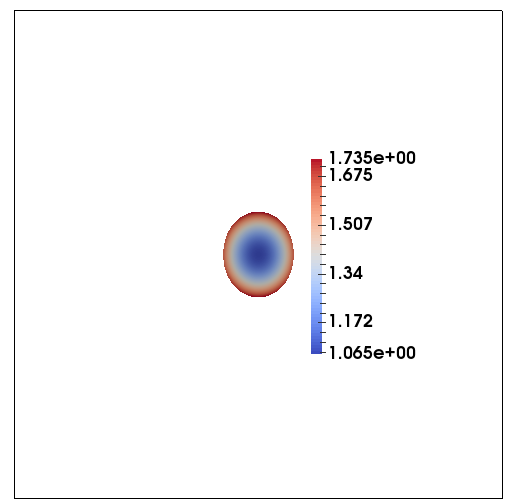}}}\hspace{-1mm}
\subfloat[t=200]{\includegraphics[height=.22\textheight,angle=0]{{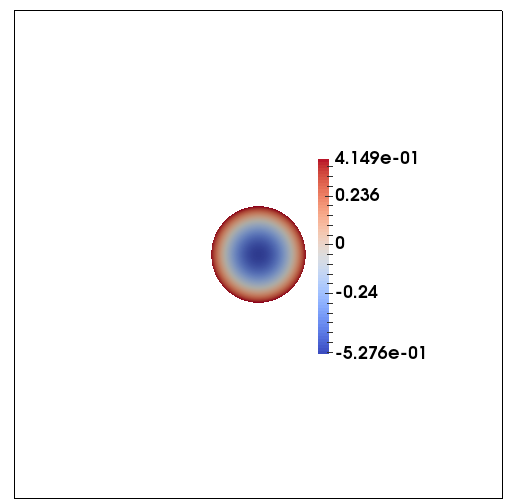}}}\hspace{0mm}
\caption{Simulations of $u_h$ with $\Omega=(-10,10)^2$, $Q = 1$, $\beta = 0.1$ and $\alpha=0.5$ (top row), $\alpha=2$ (middle row), $\alpha=7$ (bottom row).}
\label{f:avary}
\end{figure}

\begin{figure}[htbp]
\centering
\subfloat[t=0]{\includegraphics[height=.22\textheight,angle=0]{{pressure_id_big.png}}}\hspace{-1mm}
\subfloat[t=0.5]{\includegraphics[height=.22\textheight,angle=0]{{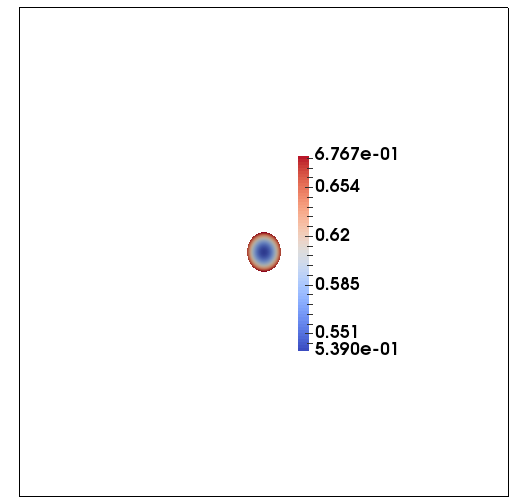}}}\hspace{-1mm}
\subfloat[t=2]{\includegraphics[height=.22\textheight,angle=0]{{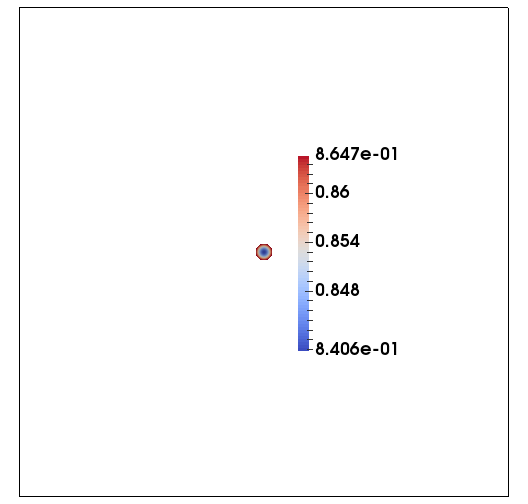}}}\hspace{0mm}
\subfloat[t=0]{\includegraphics[height=.22\textheight,angle=0]{{pressure_id_big.png}}}\hspace{-1mm}
\subfloat[t=115]{\includegraphics[height=.22\textheight,angle=0]{{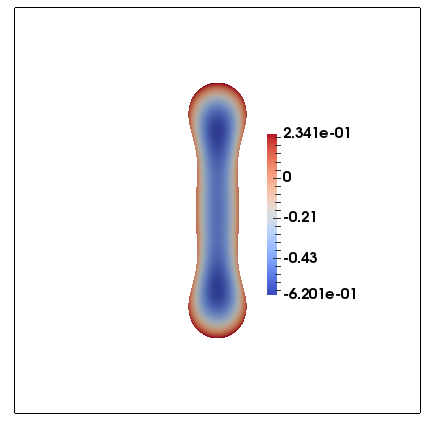}}}\hspace{-1mm}
\subfloat[t=180]{\includegraphics[height=.22\textheight,angle=0]{{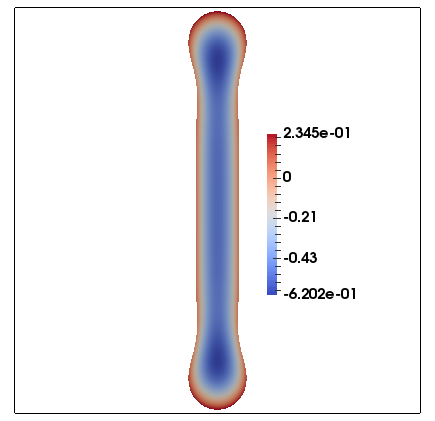}}}\hspace{0mm}
\subfloat[t=0]{\includegraphics[height=.22\textheight,angle=0]{{pressure_id_big.png}}}\hspace{-1mm}
\subfloat[t=30]{\includegraphics[height=.22\textheight,angle=0]{{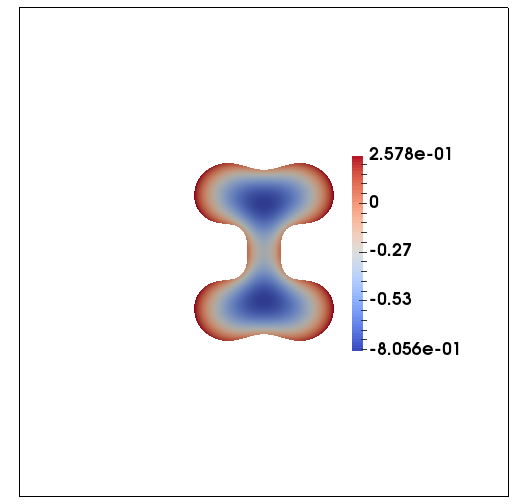}}}\hspace{-1mm}
\subfloat[t=45]{\includegraphics[height=.22\textheight,angle=0]{{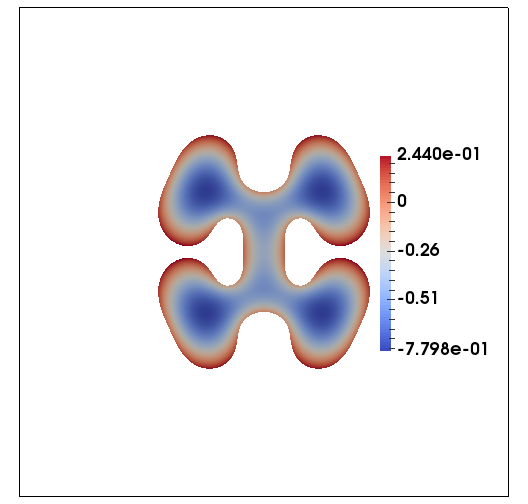}}}\hspace{0mm}
\caption{Simulations of $u_h$ with $\Omega=(-10,10)^2$, $Q = 1$, $\alpha= 1$ and $\beta=0.5$ (top row), $\beta=0.2$ (middle row), $\beta=0.1$ (bottom row).}
\label{f:bvary}
\end{figure}

\end{document}